\documentclass[11pt]{amsart}

\usepackage[top = 1in, left = 1.25in, right = 1.5in, bottom = 1in]{geometry}

\usepackage{etex}
\usepackage{xcolor,colortbl}
\usepackage[all]{xy}
\usepackage{booktabs}
\usepackage{tikz}
\usetikzlibrary{matrix}
\usepackage{array}

\usepackage{amsrefs}
\usepackage{mathrsfs}
\usepackage{verbatim}
\usepackage{graphicx}

\def\be{\begin{equation}}
\def\en{\end{equation}}

\newcommand{\htop}{h_{\text{\normalfont top}}}

\newtheorem{theorem}{Theorem}[section] 
\newtheorem{lemma}[theorem]{Lemma}     
\newtheorem{corollary}[theorem]{Corollary}
\newtheorem{proposition}[theorem]{Proposition}

\theoremstyle{definition}

\newtheorem*{ack*}{Acknowledgment}

\theoremstyle{remark}
\newtheorem{remark}[theorem]{Remark}

\numberwithin{equation}{section}

\title{Entropy on regular trees}

\author{Karl Petersen}
\address{Department of Mathematics,
	CB 3250 Phillips Hall,
	University of North Carolina,
	Chapel Hill, NC 27599 USA}
\email{petersen@math.unc.edu}

\author{Ibrahim Salama}
\address{School of Business, North Carolina Central University, Durham, NC 27707 USA}
\email{isalama@nccu.edu}

\date{\today}

\begin{document}
	

\begin{abstract}
	We show that the limit in our definition of tree shift topological entropy is actually the infimum, as is the case for both the topological and measure-theoretic entropies in the classical situation when the time parameter is $\mathbb Z$. 
	As a consequence, tree shift entropy becomes somewhat easier to work with. 
	For example, the statement that the topological entropy of a tree shift defined by a one-dimensional subshift dominates the topological entropy of the latter can now be extended from shifts of finite type to arbitrary subshifts. 
	Adapting to trees the strip method already used to approximate the hard square constant on $\mathbb Z^2$, we show that the entropy of the hard square tree shift on the regular $k$-tree {\em increases} with $k$, in contrast to the case of $\mathbb Z^k$. 
	We prove that the strip entropy approximations increase strictly to the entropy of the golden mean tree shift for $k=2,\dots,8$ and propose that this holds for all $k \geq 2$. 
	We study the dynamics of the map of the simplex that advances the vector of ratios of symbol counts as the width of the approximating strip is increased, providing a fairly complete description for the golden mean subshift on the $k$-tree for all $k$. 
	This map provides an efficient numerical method for approximating the entropies of tree shifts defined by nearest neighbor restrictions. 
	Finally, we show that counting configurations over certain other patterns besides the natural finite subtrees yields the same value of entropy for tree SFT's. 	
	\end{abstract}
	
\maketitle
\section{Introduction}
Entropy is a single number attached to a topological or measure-theoretic dynamical system that in a limited but precise way describes the complexity or richness of the system. 
In recent years increasing attention has been paid to the calculation of entropy for systems in which the ``time" is not $\mathbb Z$ nor $\mathbb R$, but perhaps $\mathbb Z^d$ for some $d \geq 2$, or an arbitrary amenable group, or even a free or arbitrary countable group. 
We will not attempt to review the extensive and rapidly developing literature here (nor the connections with information theory, statistical mechanics, and other areas), referring only to \cites{KerrLi2016, Bowen2017, HochmanMeyerovitch2010} for background and references. 

Aubrun and B\'eal \cites{AB1,AB2,AB3,AB4,AB5} proposed studying subshifts on trees, since for such systems the ``time" has both higher-dimensional and directional aspects, making them perhaps somehow between one- and higher-dimensional subshifts.
Steve Piantadosi \cite{Piantadosi2008} studied the topological entropy of the hard square model on free groups $\mathbb F_k$. He obtained an explicit formula in terms of a rapidly converging infinite series and used it to show numerically but rigorously that the entropy increases with $k$ for a range of $k$. 
Here we investigate some of these same questions for trees, with different methods but with some closely related results. 

In a previous paper \cite{PS2018} we gave a definition of entropy for tree shifts and showed that the limit in the definition exists. 
We proved that for a $2$-tree shift defined by nearest neighbor constraints, the tree-shift entropy dominates the entropy of the corresponding one-dimensional shift of finite type. 
We also provided estimates for the entropies of various $2$-tree shifts, especially the ones determined by the  ``golden mean" (or ``hard square" or ``hard core") condition that no two adjacent nodes have identical labels. 

One of our main results here (Theorem \ref{thm:lim=inf}) is that the limit in the definition is actually an infimum. As a corollary (Corollary \ref{cor:entcomp}) we show that the entropy comparison between a one-dimensional shift of finite type and the tree shift it defines holds for all subshifts. 
Then we adapt the ``strip method" used for lattice shifts \cites{Pavlov2012, MarcusPavlov2013} to study the entropy $h^{(k)}$ of the golden mean subshift on the regular $k$-tree. 
Generalizing and improving the result in \cite{PS2018}, we show in Theorem \ref{th:dimincrease} that $h^{(k)}$ is strictly increasing in $k$.
This contrasts with the apparent {\em decrease} of the entropy for the golden mean SFT's on $\mathbb Z^k$ for $k=1,2,3,4$ \cite{GamarnikKatz2009}. 
In Theorems \ref{th:striplim} and \ref{th:stripincrease} we show that for each fixed $k=2,\dots,8$ the strip entropies $h_n^{(k)}$ {increase strictly} to $h^{(k)}$. 
We believe that the statement holds for all $k \geq 2$. 
As in \cite{Piantadosi2008}, a related map of the interval (or, for the case of more general tree subshifts, simplex) appears as one considers ratios of symbol counts in the improving approximations. 
We produce a thorough analysis of this map for the case of golden mean restrictions (see Theorem \ref{thm:intervalmap}) and show in Section \ref{sec:general} how to use it to obtain rapidly converging approximations to the entropies of more general tree shifts.
Finally, we count configurations over extensions of the patterns that in this setting correspond to intervals in the one-dimensional case, showing in Corollary \ref{cor:entcomp} that for tree SFT's the resulting entropy is the same. 

Apparently the definition of entropy considered here and in \cites{Piantadosi2008,PS2018} is not the same as sofic entropy (see \cites{KerrLi2016,Bowen2017}), since the latter is a conjugacy invariant while the entropy considered here can increase under higher block codes.

\subsection{Notation and setup}\label{sec:setup}
Let $k \geq 2$ and let $\Sigma_k=\{0,1,\dots ,k-1\}$. The set $\Sigma_k^*$ of all finite words on the alphabet $\Sigma_k$ is the {\em $k$-tree}, which is naturally visualized as the Cayley graph of the free semigroup on $k$ generators. 
The empty word $\epsilon$ corresponds to the root of the tree and the neutral element of the semigroup. Let $A=\{0,1,\dots ,d-1\}$ be an alphabet on $d$ symbols. 
A {\em labeled tree} is a function $\tau : \Sigma_k^* \to A$. 
For $w \in \Sigma_k^*$, we think of $\tau(w)$ as the label attached to the node determined by $w$. 
For each $n \geq 0$ let $\Delta_n= \cup_{0 \leq i \leq n} \Sigma_k^i$ denote the initial height-$n$ subtree of the $k$-tree. 
The cardinality of $\Delta_n$ is $|\Delta_n|=1+k+ \dots +k^n$.
An {\em $n$-block} is a function $B: \Delta_n \to A$, which we think of as a labeling of $\Delta_n$ or a configuration on $\Delta_n$. 
We say that an $n$-block $B$ {\em appears} in a labeled tree $\tau$ if there is a node $x \in \Sigma_k^*$ such that $\tau(xw)=B(w)$ for all $w \in \Delta_n$. 
A {\em tree shift} $Z$ is the set of all labeled trees which omit all of a certain set (possibly infinite) of forbidden blocks. These are exactly the closed shift-invariant subsets of the full tree shift space $T(A)=A^{\Sigma_k^*}$. 
A tree shift $Z$ is called {\em transitive} if it contains a labeled tree $\tau$ such that for every $\xi \in Z$ every block that appears in $\xi$ also appears in $\tau$. 
Such a labeled tree $\tau$ is called a {\em transitive point}. 

The {\em complexity function} $p_\tau$ of the labeled tree $\tau$ assigns to each $n \geq 0$ the number of $n$-blocks that appear in $\tau$. 
The {\em complexity function} $p_Z(n)$ of a tree shift $Z$ gives for each $n \geq 0$ the number of $n$-blocks among all labeled trees in $Z$. 
We are interested in studying the complexity functions of trees that are labeled according to certain restrictions, in particular nearest-neighbor constraints specified by $d$-dimensional $0,1$ transition matrices.
 In \cite{PS2018} it was proved that for any labeled tree $\tau$ the limit
\be
h(\tau) = \lim_{n \to \infty} \frac{\log p_\tau (n)}{1+k+ \dots +k^n}
\en
exists. 
This limit is called the {\em topological entropy} of the labeled tree $\tau$. 
The topological entropy $h(Z)$ of a transitive tree shift $Z$ is defined to be the topological entropy of any of its transitive points.

\section{The limit in the definition of tree shift entropy is the infimum}
 
 \begin{theorem}\label{thm:lim=inf} 
 	The limit in the definition of tree shift topological entropy is actually the infimum:
 \be
 h(\tau) = \lim_{n \to \infty} \frac{\log p_\tau (n)}{1+k+ \dots +k^n}= 
 \inf \left\{\frac{\log p_\tau (n)}
 	{1+k+ \dots +k^n} : n \in \mathbb N\right\}.
 \en
 \end{theorem}
 \begin{proof}
 		 In the proof of the existence of the limit for $h$ in \cite{PS2018}, $\Delta_{jm}$ was decomposed into a union of shifts of $\Delta_m$. But these copies of $\Delta_m$ did not have independent entries; in fact they were not disjoint, since the last row of one formed the root vertices of the next ones. So here we improve the estimate 2.2 in \cite{PS2018} by making them disjoint. 

Fix $n \geq 1$ and consider $\Delta_n$. 
Its last row has $k^n$ entries. The next row has $k^{n+1}$ entries, and we use these as root vertices for new shifts of $\Delta_n$. This new row ends with $k^{n+1} k^n$ entries, so the next row has $k^{2n+2}$ entries, which we use as vertices of new shifts of $\Delta_n$. The last row now has $k^{2n+2} k^n=k^{3n+2}$ entries, the next row has $k^{3n+3}$ and each of these becomes the root of a new shift of $\Delta_n$. 

We have formed a $\Delta_{4(n+1)-1}$ out of $1 + k^{n+1} + k^{2(n+1)} + k^{3(n+1)}$ disjoint copies of $\Delta_n$. In general, for each $j \geq 1$ we have
\be
p(j(n+1)-1) \leq p(n)^{(k^{j(n+1)}-1)/(k^{n+1}-1)}.
\en
 (In formula 2.2 of \cite{PS2018} we find the same estimate, except with $\leq p(n)$ replaced by $\leq p(n+1)$. So this estimate is better.)

Then take logarithms, divide by $(k^{j(n+1)} -1)/(k-1)$, and
take the limit as $j \to \infty$, to find that
\be
 h \leq \frac{\log p(n)}{1+k+k^2+ \dots + k^n}. 
 \en
 Therefore 
 \be
 h=\inf_n \frac{\log p(n)} {1+k+k^2+\dots +k^n}.
 \en 
 \end{proof}

 A first consequence of this result is a generalization of Theorem 3.3 of \cite{PS2018} for $2$-trees from shifts of finite type to arbitrary subshifts. 
 Using the same argument and Theorem \ref{th:dimincrease} below, the statement extends to $k$-trees. 
 Let $M$ be a $d \times d$ matrix with entries from $\{0,1\}$. The matrix $M$ defines a {\em one-step shift of finite type} (SFT) $X_M$ on the alphabet $A$ and a {\em tree shift} $Z_M$ consisting of all $k$-trees labeled by $A$ with the property that for every $w=w_0 w_1 \dots w_j \in \Sigma_k^*$, $M_{w_iw_{i+1}}=1$ for all $0 \leq i < j$. 
 In \cite{PS2018}*{Theorem 3.3} it was also proved that $h(Z_M)=\sup\{h(\tau): \tau \in Z_M\}$ dominates the entropy $h(X_M)$ of the associated shift of finite type. 
 
 More generally, given any subshift $X \subset A^{\mathbb Z}$, there is a naturally associated tree shift $Z(X)$ defined as follows. 
 Denote by $\mathcal L(X)$ the {\em language} of $X$, namely the set of all finite words on $A$ found as subwords of sequences in $X$. 
 The $k$ shifts on $\Sigma_k^*$  are defined by $\sigma_s (w) = sw, w \in \Sigma_k^*, s\in \Sigma_k$. 
 For $s_1, s_2, \dots , s_m \in \Sigma_k$ and $s=s_1 s_2 \dots s_m \in \Sigma_k^*$, define 
 {$\sigma_s= \sigma_{s_1} \dots \sigma_{s_m}$.}
 On a labeled tree $\tau$, define $(\sigma_s \tau )(w) = \tau(\sigma_s w) =\tau(sw), s \in \Sigma_k^*$.
 
 We define the {\em one-dimensional language} of a tree shift $Z$ to be the set $\mathcal L^{(1)}(Z)$ of strings on the alphabet $A$ found along paths in the tree:
 \be
 \begin{gathered}
 \mathcal L^{(1)}(Z) = \epsilon \cup \{\tau (s): s \in \Sigma_k^*\} 
  \cup \{\tau (w) \tau({w s_1 }) \tau (w s_1 s_2) \dots \tau (w s_1 s_2 \dots s_m):\\
  m \in \mathbb N; w \in \Sigma_k^*; s_1,s_2,  \dots, s_m \in \Sigma_k\}.
 \end{gathered}
 \en
 Given a subshift $X \subset A^{\mathbb Z}$, we define the {\em tree shift associated to $X$} to be
 the unique tree shift $Z(X)$ such that
 \be
 \mathcal L^{(1)}(Z(X)) = \mathcal L(X).
 \en

 \begin{corollary}\label{cor:entcomp}
 	Let $X \subset A^{\mathbb Z}$ be a subshift on a finite alphabet, let $k = 2$, and let $Z(X)$ be the tree shift on the binary tree associated with $X$. Then
 	\be
 	\htop (X) \leq h(Z(X)).
 	\en
 
 \end{corollary}
\begin{proof}
 Given any subshift $X$, for each $r \in \mathbb N$ let $X_r$ be the shift of finite type which has the same $r$-blocks as $X$. Then $X$ is the decreasing intersection of the $X_r$. 
  Denote by $p_X$, $p_{X_r}$, and $p_Z$ the complexity functions of the subshifts $X$ and $X_r$ and the tree shift $Z$, respectively.
 Then for $r \geq n$ we have $p_{X_r}(n)=p_X(n)$, and similarly for $p_{Z(X_r)}$ and $p_{Z(X)}$. Thus 
 \be\label{eq:infs}
 \begin{aligned}
 \htop(X) &\leq \lim_{r \to \infty} \htop (X_r) = \inf_r \htop(X_r)\\
          &=\inf_r \inf_n \frac{1}{n} \log p_{X_r}(n) = \inf_n \inf_r \frac{1}{n} \log p_{X_r}(n)\\
          &= \inf_n \frac{1}{n} \log p_{X}(n) =\htop(X).
          \end{aligned}
 \en
 
In \cite{PS2018}*{Theorem 3.3} it was proved that if $X$ is a one-step SFT, then
\be
\htop (X) \leq h(Z(X)).
\en
Each $X_r$ is an $(r-1)$-step SFT and is topologically conjugate to a one-step SFT $Y_r$ on the alphabet $A^{(r)}(X)$ of $r$-blocks which appear in $X$. 
In a labeled tree in $Z(Y_r)$, we think of the {\em last} entry of the $r$-block labeling a node as being attached to that node. 
A labeling by elements of $A$ of the vertices of the $k$-tree is consistent with the restrictions from $X_r$ if and only if it is consistent with the restrictions from $Y_r$, so 
$p_{Z(X_r)}=p_{Z(Y_r)}$. 
Using (\ref{eq:infs}), \cite{PS2018}*{Theorem 3.3} applied to $Y_r$ and $Z(Y_r)$, 
and Theorem \ref{thm:lim=inf}, we then have 
 \be
 \begin{aligned}
 \htop(X) &= \inf_r \htop (X_r) = \inf_r \htop (Y_r) \leq \inf_r h(Z(Y_r)) =\inf_r h(Z(X_r))\\
          &= \inf_r \inf_n \frac{1}{n} \log p_{Z(X_r)}(n)  = \inf_n \frac{1}{n}\inf_r \log p_{Z(X_r)}(n) \\
          &= \inf_n \frac{1}{n} \log p_{Z(X)}(n) = h(Z(X)).
           	\end{aligned}
           \en
   \end{proof}        
 
 \begin{remark}
 	A similar statement applies to the entropy of an invariant measure $\mu$ on  a tree shift $Z$ and can be proved ny a similar argument. 
 	Denote by $\alpha_n$ the partition of $Z$ according to all possible labelings of the nodes of $\Delta_n$ 
 	by elements of the alphabet $A$. Then
 	\be
 	\begin{aligned}
 	h_\mu (Z) &= \lim_{n \to \infty} \frac{1}{|\Delta_n|} H_\mu(\alpha_n) = \lim_{n \to \infty} \frac{1}{|\Delta_n|} -\sum_{\Lambda \in \alpha_n} \mu(\Lambda) \log \mu (\Lambda)\\
    &=\inf_n \frac{1}{|\Delta_n|} H_\mu(\alpha_n).
    \end{aligned}
\en
    \end{remark}

\section{Strict increase with dimension}
	We use the strip method for the golden mean shift of finite type on the $k$-tree, $k \geq 2$, to show that the entropy $h^{(k)}$ is strictly increasing in $k$.

For each $n=0,1,2,\dots$ we define a $1$-dimensional SFT $\Sigma_n^{(k)}$ whose alphabet consists of the legal (no adjacent $1$'s) labelings of {the subtree with $k^n$ nodes consisting of a vertex with $\Delta_{n-1}^{(k)}$ attached.}
 ($\Delta_{-1}^{(k)}=\emptyset$.)
{$\Sigma_0^{(k)}$ consists of allowed labelings of the left edge $0^*$ (see Section \ref{sec:setup}) of the $k$-tree and is conjugate to the ordinary one-dimensional golden mean SFT. 
	$\Sigma_1^{(k)}$ is the one-dimensional SFT formed by the strip of adjacent pairs of vertices $0^i,0^i1$ down the left edge. 
	For $k=2$, thinking of the possible labels on vertices as $0$ and $1$, we may represent it as the set of all one-sided sequences on an alphabet $\{a=00,b=01,c=10\}$, with the restriction that the block $cc$ does not occur.
	The alphabet for $\Sigma_2^{(2)}$ consists of the allowed labelings of a $4$-vertex tree as in Figure \ref{fig:2strip}, etc.} 
	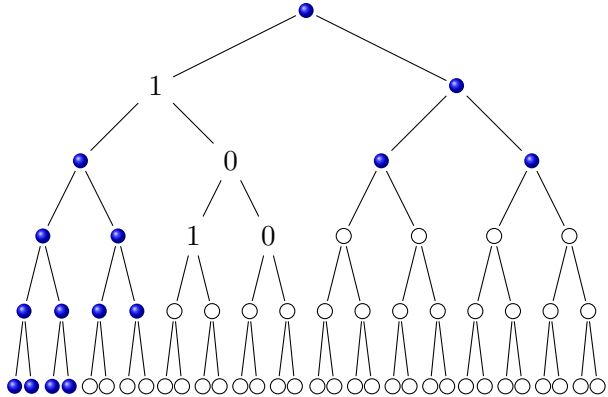
\begin{figure}[h]
				\begin{tikzpicture}[scale=1]
			\node at (0,0) (v) {};
			\shade [ball color=blue] (v) circle (0.1);
			\node at (-2,-1) (v0) {1};
			\node at (2,-1)  (v1) {};
			\shade [ball color=blue]  (v1) circle (0.1);
			\node at (-3,-2) (v00) {};
			\shade [ball color=blue]  (v00) circle (0.1);
			\node at (-1,-2) (v01) {0};
			\node at (1,-2) (v10) {};
			\shade [ball color=blue]  (v10) circle (0.1);
			\node at (3,-2) (v11) {};
			\shade [ball color=blue]  (v11) circle (0.1);
			\node at (-3.5,-3) (v000) {};
			\shade [ball color=blue]  (v000) circle (0.1);
			\node at (-2.5,-3) (v001) {};
			\shade [ball color=blue]  (v001) circle (0.1);
			\node at (-1.5,-3) (v010) {1};
			\node at (-.5,-3) (v011) {0};
			\node at (.5,-3) (v100) {};
			\draw  (v100) circle (0.1);
			\node at (1.5,-3) (v101) {};
			\draw  (v101) circle (0.1);
			\node at (2.5,-3) (v110) {};
			\draw  (v110) circle (0.1);
			\node at (3.5,-3) (v111) {};
			\draw  (v111) circle (0.1);
			\node at (-3.75,-4) (v0000) {};
			\shade [ball color=blue]  (v0000) circle (0.1);
			\node at (-3.25,-4) (v0001) {};
			\shade [ball color=blue]  (v0001) circle (0.1);
			\node at (-2.75,-4) (v0010) {};
			\shade [ball color=blue]  (v0010) circle (0.1);
			\node at (-2.25,-4) (v0011) {};
			\shade [ball color=blue]  (v0011) circle (0.1);
			\node at (-1.75,-4) (v0100) {};
			\draw  (v0100) circle (0.1);
			\node at (-1.25,-4) (v0101) {};
			\draw  (v0101) circle (0.1);
			\node at (-.75,-4) (v0110) {};
			\draw  (v0110) circle (0.1);
			\node at (-.25,-4) (v0111) {};
			\draw  (v0111) circle (0.1);
			
			\node at (.25,-4) (v1000) {};
			\draw  (v1000) circle (0.1);
			\node at (.75,-4) (v1001) {};
			\draw  (v1001) circle (0.1);
			
			\node at (1.25,-4) (v1010) {};
			\draw  (v1010) circle (0.1);
			\node at (1.75,-4) (v1011) {};
			\draw  (v1011) circle (0.1);
			
			\node at (2.25,-4) (v1100) {};
			\draw  (v1100) circle (0.1);
			\node at (2.75,-4) (v1101) {};
			\draw  (v1101) circle (0.1);
			
			\node at (3.25,-4) (v1110) {};
			\draw  (v1110) circle (0.1);
			\node at (3.75,-4) (v1111) {};
			\draw  (v1111) circle (0.1);
			
			\node at (-3.875,-5) (v00000) {};
			\shade [ball color=blue]  (v00000) circle (0.1);
			\node at (-3.65,-5) (v00001) {};
			\shade [ball color=blue]  (v00001) circle (0.1);
			
			\node at (-3.375,-5) (v00010) {};
			\shade [ball color=blue]  (v00010) circle (0.1);
			\node at (-3.15,-5) (v00011) {};
			\shade [ball color=blue]  (v00011) circle (0.1);
			
			\node at (-2.875,-5) (v00100) {};
			\draw  (v00100) circle (0.1);
			\node at (-2.65,-5) (v00101) {};
			\draw  (v00101) circle (0.1);
			
			\node at (-2.375,-5) (v00110) {};
			\draw  (v00110) circle (0.1);
			\node at (-2.125,-5) (v00111) {};
			\draw  (v00111) circle (0.1);
			
			\node at (-1.875,-5) (v01000) {};
			\draw  (v01000) circle (0.1);
			\node at (-1.65,-5) (v01001) {};
			\draw  (v01001) circle (0.1);
			
			\node at (-1.375,-5) (v01010) {};
			\draw  (v01010) circle (0.1);
			\node at (-1.15,-5) (v01011) {};
			\draw  (v01011) circle (0.1);
			
			\node at (-.875,-5) (v01100) {};
			\draw  (v01100) circle (0.1);
			\node at (-.65,-5) (v01101) {};
			\draw  (v01101) circle (0.1);
			
			\node at (-.375,-5) (v01110) {};
			\draw  (v01110) circle (0.1);
			\node at (-.15,-5) (v01111) {};
			\draw  (v01111) circle (0.1);

			\draw (v) -- (v0);
			\draw (v) -- (v1);
			\draw (v0) -- (v00);
			\draw (v0) -- (v01);
			\draw (v1) -- (v10);
			\draw (v1) -- (v11);
			\draw (v00) -- (v000);
			\draw (v00) -- (v001);
			\draw (v01) -- (v010);
			\draw (v01) -- (v011);
			\draw (v10) -- (v100);
			\draw (v10) -- (v101);
			\draw (v11) -- (v110);
			\draw (v11) -- (v111);
			\draw (v000) -- (v0000);
			\draw (v000) -- (v0001);
			\draw (v001) -- (v0010);
			\draw (v001) -- (v0011);
			\draw (v010) -- (v0100);
			\draw (v010) -- (v0101);
			\draw (v011) -- (v0110);
			\draw (v011) -- (v0111);
			
			\draw (v100) -- (v1000);
			\draw (v100) -- (v1001);
			
			\draw (v101) -- (v1010);
			\draw (v101) -- (v1011);
			
			\draw (v110) -- (v1100);
			\draw (v110) -- (v1101);
			
			\draw (v111) -- (v1110);
			\draw (v111) -- (v1111);
			
			\draw (v0000) -- (v00000);
			\draw (v0000) -- (v00001);
			
			\draw (v0001) -- (v00010);
			\draw (v0001) -- (v00011);
			
			\draw (v0010) -- (v00100);
			\draw (v0010) -- (v00101);
			\draw (v0011) -- (v00110);
			\draw (v0011) -- (v00111);
			\draw (v0100) -- (v01000);
			\draw (v0100) -- (v01001);
			\draw (v0101) -- (v01010);
			\draw (v0101) -- (v01011);
			\draw (v0110) -- (v01100);
			\draw (v0110) -- (v01101);
			\draw (v0111) -- (v01110);
			\draw (v0111) -- (v01111);
			
			
			\node at (.15,-5) (v10000) {};
			\draw  (v10000) circle (0.1);
			\node at (.375,-5) (v10001) {};
			\draw  (v10001) circle (0.1);
			
			\node at (.65,-5) (v10010) {};
			\draw  (v10010) circle (0.1);
			\node at (.875,-5) (v10011) {};
			\draw  (v10011) circle (0.1);
			
			\node at (1.15,-5) (v10100) {};
			\draw  (v10100) circle (0.1);
			\node at (1.375,-5) (v10101) {};
			\draw  (v10101) circle (0.1);
			
			\node at (1.65,-5) (v10110) {};
			\draw  (v10110) circle (0.1);
			\node at (1.875,-5) (v10111) {};
			\draw  (v10111) circle (0.1);
			
			\node at (2.125,-5) (v11000) {};
			\draw  (v11000) circle (0.1);
			\node at (2.375,-5) (v11001) {};
			\draw  (v11001) circle (0.1);
			
			\node at (2.65,-5) (v11010) {};
			\draw  (v11010) circle (0.1);
			\node at (2.875,-5) (v11011) {};
			\draw  (v11011) circle (0.1);
			
			\node at (3.15,-5) (v11100) {};
			\draw  (v11100) circle (0.1);
			\node at (3.375,-5) (v11101) {};
			\draw  (v11101) circle (0.1);
			
			\node at (3.65,-5) (v11110) {};
			\draw  (v11110) circle (0.1);
			\node at (3.875,-5) (v11111) {};
			\draw  (v11111) circle (0.1);

			\draw (v1000) -- (v10000);
			\draw (v1000) -- (v10001);
			
			\draw (v1001) -- (v10010);
			\draw (v1001) -- (v10011);
			
			\draw (v1010) -- (v10100);
			\draw (v1010) -- (v10101);
			\draw (v1011) -- (v10110);
			\draw (v1011) -- (v10111);
			\draw (v1100) -- (v11000);
			\draw (v1100) -- (v11001);
			\draw (v1101) -- (v11010);
			\draw (v1101) -- (v11011);
			\draw (v1110) -- (v11100);
			\draw (v1110) -- (v11101);
			\draw (v1111) -- (v11110);
			\draw (v1111) -- (v11111);
						\end{tikzpicture}
		\caption{$\Sigma_2^{(2)}$}
		\label{fig:2strip}
	\end{figure}

{Denote by $h(\Sigma_n^{(k)})$ the topological entropy of the one-dimensional SFT $\Sigma_n^{(k)}$.} 
{Recall that the alphabet of $\Sigma_n^{(k)}$ consists of allowed labelings of $k^n$ sites.}
We will show that the {\em site specific entropies} $h_n^{(k)}=h(\Sigma_n^{(k)})/k^n$ of the strip SFT's $\Sigma_n^{(k)}$ have limit $h^{(k)}$=the entropy of the golden mean labeled $k$-tree. 
Moreover, with the help of computer algebra, we can prove that for each $k=2,\dots,8$ the $h_n^{(k)}$ are {\em strictly increasing}, and we believe that this holds for all $k \geq 2$. 

For each $n=0,1,\dots$, denote by $B_n^{(k)}$ the number of different labelings of $\Delta_n^{(k)}$ by the alphabet $\{ 0,1\}$, by $B_n^{(k)}(0)$ the number of such labelings that have $0$ at the root, and by $r_n^{(k)}=r_n^{(k)}(0)$ the ratio $B_n^{(k)}(0)/B_n^{(k)}$.

While $k$ is fixed, we will suppress the exponents ${(k)}$. 

The labeling counts satisfy the recursions
\be
\begin{gathered}
B_0(0)=B_0(1)=1, B_0=2, r_0=1/2, \\
B_{n+1}(0)=[B_n(0)+B_n(1)]^k=B_n^k,B_{n+1}(1)=(B_n(0))^k.
\end{gathered}
\en
{We define $B_{-1}(0)=1, B_{-1}(1)=0$, so that $B_{-1}=1$.}
{
The labels at the nodes $i$ on the left edge respect the SFT restriction no $11$.
Let $a_i(0)$ denote the number of ways to label the first $i$ levels of the strip of width $n \geq 0$ with $0$ labeling node $i$, and define $a_i(1)$ similarly.
For $n \geq 1$ the $\Delta_{n-1}$ attached to node $i+1$ can be labeled in $B_{n-1}$ ways if node $i+1$ has label $0$, and in $B_{n-1}(0)$ ways if node $i+1$ has label $1$.}
Therefore
\be
\begin{bmatrix} 
a_{i+1}(0) \\
a_{i+1}(1) 
\end{bmatrix}
=
\begin{bmatrix} 
B_{n-1}^{k-1} & B_{n-1}^{k-1} \\
B_{n-1}(0)^{k-1} & 0 
\end{bmatrix} 
\begin{bmatrix} 
a_{i}(0) \\
a_{i}(1) 
\end{bmatrix} .
\en
{Thus for $n \geq 0$ the number $a_i(0)+a_i(1)$ of $i$-blocks in $\Sigma_n^{(k)}$ has the same asymptotic growth rate as the entries of $i$'th powers of the matrix}
\be
\begin{bmatrix} 
	B_{n-1}^{k-1} & B_{n-1}^{k-1} \\
	B_{n-1}(0)^{k-1} & 0 
\end{bmatrix} .
\en
The latter is given by powers of the Perron-Frobenius eigenvalue $\lambda_{n-1}$, which is the maximum root of the characteristic equation
\be
{f_{n-1}(\lambda)= \lambda^2 - B_{n-1}^{k-1} \lambda - [B_{n-1} B_{n-1}(0)]^{k-1}}, 
\en
namely
\be
{\lambda_{n-1}= \frac{B_{n-1}^{k-1}+\sqrt{B_{n-1}^{2k-2}+4[B_{n-1}B_{n-1}(0)]^{k-1}}}{2} = B_{n-1}^{k-1} c_{n-1},}
\en
with
\be
{c_n=(1+\sqrt{1+4r_{n}(0)^{k-1}})/{2}.}
\en
Thus
{$c_n=\lambda_n/B_{n}^{k-1}=c_n(r)$ is the largest root of $x^2-x-r_{n}^{k-1}$
 and satisfies $c_n^2-c_n=r_{n}^{k-1}$.}
 
 	We have 
 	\be
 	{h_n^{(k)}=\frac{\log \lambda_{n-1}}{k^n}.}
 	\en
 	Note that all $h_n^{(k)} \leq \log k$ because $B_n \leq k^{k^n}$. 

\begin{theorem}\label{th:striplim}
	For each fixed $k=2,3,\dots$ the site specific strip approximation entropies $h_n^{(k)}$ converge to the entropy $h^{(k)}$ of the golden mean SFT on the $k$-tree.
\end{theorem}
\begin{proof}
Since $c_n$ is bounded,  
\be 
\begin{aligned}
\lim h_n^{(k)} &= \lim \frac{\log \lambda_{n-1}}{k^n} 
= \lim \frac{\log B_{n-1}^{k-1}}{k^n} + \lim \frac{\log c_{n-1}}{k^n} \\
&= \lim \frac {(k-1)\log B_{n-1}}{k^n-1} = h^{(k)}.
\end{aligned}
		\en
\end{proof}

We develop for golden mean tree shifts the analogue of Piantadosi's \cite{Piantadosi2008} infinite series formula for the entropy of the golden mean SFT on a free group. 
Then we use the formula to prove that the entropy $h^{(k)}$ of the golden mean SFT on the $k$-tree is strictly increasing in $k$, for all $k \geq 2$. 
(Piantadosi used his infinite series formula and rigorous numerical estimates to prove the strict increase for a range of $k$.) 
Our infinite series formula follows from two easy lemmas.

\begin{lemma}
\be
h_{n+1}^{(k)}=h_n^{(k)}+\frac{1}{k^{n+1}}\log(1+r_{n-1}^k)^{k-1}+\frac{1}
{k^{n+1}}\log(c_{n})-\frac{1}{k^n}\log(c_{n-1}).
\en
\end{lemma}
\begin{proof}
\be
\begin{aligned}
h_{n+1} & =  \frac{1}{k^{n+1}}\log(\lambda_{n})  =  \frac{1}{k^{n+1}}\log(B_{n}^{k-1}c_{n})  \\
& =  \frac{1}{k^{n+1}}\log[B_{n-1}^{k(k-1)}(1+r_{n-1}^k)^{k-1}c_{n}]  \\
& =  \frac{1}{k^{n+1}}\log[B_{n-1}^{k(k-1)}c_{n-1}^{{k}}\frac{(1+r_{n-1}^k)^{k-1}c_{n}}
{c_{n-1}^k}]  \\
& =  \frac{1}{k^n}\log(B_{n-1}^{k-1}c_{n-1})+\frac{1}{k^{n+1}}\log(1+r_{n-1}^k)^{k-1}
+\frac{1}{k^{n+1}}\log(c_{n})-\frac{1}{k^n}\log(c_{n-1})  \\
& =  h_n^{(k)}+\frac{1}{k^{n+1}}\log(1+r_{n-1}^k)^{k-1}+\frac{1}{k^{n+1}}
\log(c_{n})-\frac{1}{k^n}\log(c_{n-1}). 
\end{aligned}
\en
\end{proof}

Note
\be
\begin{aligned}
h_1^{(k)} & =  \frac{1}{k}\log(B_0^{k-1}c_0)  \\
& =  \frac{1}{k}\log(2^{k-1})+\frac{1}{k}\log(c_0)  \\
& =  \frac{k-1}{k}\log 2+\frac{1}{k}\log(c_0) 
\end{aligned}
\en
and
\be
\begin{aligned}
h_2^{(k)} & =  h_1^{(k)}+\frac{1}{k^2}\log(1+r_0^k)^{k-1}+\frac{1}{k^2}
\log(c_1)-\frac{1}{k}\log(c_0)  \\
& =  \frac{k-1}{k}\log 2+\frac{1}{k}\log(c_0)+\frac{1}{k^2}
\log(1+r_0^k)^{k-1}+\frac{1}{k^2}\log(c_1)-\frac{1}{k}\log(c_0)  \\
& =  \frac{k-1}{k}\log 2+\frac{1}{k^2}\log(1+r_0^k)^{k-1}+\frac{1}{k^2}
\log(c_1). 
\end{aligned}
\en

\begin{lemma}
\be
h_n^{(k)}=\frac{k-1}{k}\log 2+\sum_{i=1}^{n-1}\frac{1}{k^{i+1}}
\log(1+r_{i-1}^k)^{k-1}+\frac{1}{k^n}\log(c_{n-1}).
\en
\end{lemma}
\begin{proof}
 Assume the formula holds for $n-1$. Then
\be
h_{n-1}^{(k)}=\frac{k-1}{k}\log 2+\sum_{i=1}^{n-2}\frac{1}{k^{i+1}} \log(1+r_{i-1}^k)^{k-1}+
\frac{1}{k^{n-1}}\log(c_{n-2}),
\en
and
\be
\begin{aligned}
h_n^{(k)} & =  h_{n-1}^{(k)}+\frac{1}{k^n}\log(1+r_{n-2}^k)^{k-1}+
\frac{1}{k^n}\log(c_{n-1})-\frac{1}{k^{n-1}}\log(c_{n-2})  \\
& =  \frac{k-1}{k}\log 2+\sum_{i=1}^{n-1}\frac{1}{k^{i+1}}
\log(1+r_{i-1}^k)^{k-1}+ \frac{1}{k^n}\log(c_{n-1}). 
\end{aligned}
\en
\end{proof}

\begin{theorem}\label{th:infseries}
	For all $k \geq 2$ the entropy of the golden mean SFT on the $k$-tree is given by the following formula:
	\be
	h^{(k)}=\frac{k-1}{k}\log 2+(k-1)\sum_{i=1}^{\infty}\frac{1}{k^{i+1}}\log(1+r_{i-1}^k).
	\en
\end{theorem}

Examination of this infinite series formula provides useful upper and lower bounds for $h^{(k)}$.

\begin{lemma}
	For each $k \geq 2$, let
	\be\label{eq:L1}
	L(k)= \frac{k-1}{k}\log2+\frac{k-1}{k^2}\log(1+r_0^k)+
	\frac{k-1}{k^3}\log(1+r_1^k)
	\en
	and
	\be\label{eq:U1}
	U(k)= \frac{k-1}{k}\log 2+\frac{k-1}{k^2}\log(1+r_0^k)+
	\frac{1}{k^2}\log(1+r_1^k).
	\en
	Then
	\be
	{L(k) < h^{(k)} < U(k).}\en
\end{lemma}
\begin{proof}
	Formula (\ref{eq:L1}) is immediate by truncation of the series.
To prove Formula (\ref{eq:U1}), recall that $r_0=1/2, r_1=1/(1+1/2^k)$, and note that $r_1 > r_i$ for all $i > 1$ (see the proof of Theorem \ref{thm:intervalmap}). 
\end{proof}

These formulas give us 
\be
U(k)-L(k)=\frac{1}{k^3}\log(1+r_1^k).
\en
We also need the estimate $U(k-1)<L(k)$.

\begin{lemma}
	For each $k \geq 2$, 
	\be\label{eq:L2}
	L(k)=\frac{k-1}{k}\log 2+\frac{k-1}{k^3}\log[1+(1+r_0^k)^k]
	\en
	and
	\be\label{eq:U2}
	U(k)=\frac{k-1}{k}\log 2+\frac{1}{k^2}\log[1+(1+r_0^k)^k]
	-\frac{1}{k^2}\log(1+r_0^k).
	\en
\end{lemma} 
\begin{proof}
	\be\begin{aligned}
	L(k) & =  \frac{k-1}{k}\log 2+\frac{k-1}{k^2}\log(1+r_0^k)+
	\frac{k-1}{k^3}\log(1+r_1^k)  \\
	& =  \frac{k-1}{k}\log 2+\frac{k-1}{k^2}\log(1+r_0^k)+
	\frac{k-1}{k^3}\log[1+\frac{1}{(1+r_0^k)^k}]  \\
	& =  \frac{k-1}{k}\log 2+\frac{k-1}{k^2}\log(1+r_0^k)+
	\frac{k-1}{k^3}\log[\frac{1+(1+r_0^k)^k}{(1+r_0^k)^k}]  \\
	& =  \frac{k-1}{k}\log 2+\frac{k-1}{k^3}\log[1+(1+r_0^k)^k].
	\end{aligned}
\en
	\be
	\begin{aligned}
		U(k)&= \frac{k-1}{k}\log 2+\frac{k-1}{k^2}\log(1+r_0^k)+
		\frac{1}{k^2}\log(1+r_1^k)  \\
		& = \frac{k-1}{k}\log 2+\frac{k-1}{k^2}\log(1+r_0^k)+
		\frac{1}{k^2}\log\left[\frac{1+(1+r_0^k)^k}{(1+r_0^k)^k}\right]  \\
		& =  \frac{k-1}{k}\log 2+\frac{k-1}{k^2}\log(1+r_0^k)+
		\frac{1}{k^2}\log([1+(1+r_0^k)^k]-\frac{1}{k^2}\log(1+r_0^k)^k  \\
		& =  \frac{k-1}{k}\log 2+\frac{1}{k^2}\log[1+(1+r_0^k)^k]
		-\frac{1}{k^2}\log(1+r_0^k).
		\end{aligned}
	\en
\end{proof}

\begin{theorem}\label{th:dimincrease}
	The entropy $h^{(k)}$ of the golden mean SFT on the $k$-tree is strictly increasing in $k$.
\end{theorem}
\begin{proof}
	The inequality is verified in Corollary \ref{cor:strict} by direct computation for $2 \leq k \leq 5$. 
	{Recall that $r_0=1/2$.}
		It is enough to show that for all $k>5$ we have $L(k)>U(k-1)$, i.e.
	\be\label{eq:9}
	\begin{gathered}
	\frac{k-1}{k}\log 2+\frac{k-1}{k^3}\log[1+(1+\frac{1}{2^k})^k]  > \\ 
	\frac{k-2}{k-1}\log 2+\frac{1}{(k-1)^2}\log[1+(1+\frac{1}{2^{k-1}})^{k-1}]-
		 \frac{1}{(k-1)^2}\log(1+\frac{1}{2^{k-1}}).
	 \end{gathered}
	 \en
	Letting $x_k=1+1/2^k$ and exponentiating, (\ref{eq:9}) is equivalent to
	\be\label{eq:4est}
	2^{\frac{k-1}{k}}(1+x_k^k)^{\left(\frac{k-1}{k}\right)^3}>\frac{1+x_{k-1}^{k-1}}{x_{k-1}}.
		\en
		 We claim that
		 \be\label{eq:5est}
		 2^{\frac{k-1}{k}+\left(\frac{k-1}{k}\right)^3}> 1+x_{k-1}^{k-1} \quad\text{ for } k \geq 6,
	 \en
	 from which (\ref{eq:4est}) will follow, because then
	 \be
	 	2^{\frac{k-1}{k}}(1+x_k^k)^{\left(\frac{k-1}{k}\right)^3}>
	 	2^{\frac{k-1}{k}+\left(\frac{k-1}{k}\right)^3}> 1+x_{k-1}^{k-1} >
	 	\frac{1+x_{k-1}^{k-1}}{x_{k-1}}.
	 \en
	 Note that the left side of (\ref{eq:5est}) increases (to 4) in $k$, while $1+x_{k-1}^{k-1}$ decreases (to 2), because, using the alternating series for $\log (1+t)$,
	 \be
	 \frac{k}{2^k}-\frac{1}{2}\frac{k}{2^{2k}}< \log (1+\frac{1}{2^k})^k<\frac{k}{2^k},
	 \en
	 so that $x_k^k>x_{k+1}^{k+1}$ follows from
	 \be
	 \frac{k+1}{2^{k+1}}<k\left(\frac{1}{2^k}-\frac{1}{2}\frac{1}{2^{2k}}\right),
	 \en
	 equivalently
	 \be
	 k>1+\frac{k}{2^k},
	 \en
	 which is true for $k \geq 2$,
	 because $k/2^k$ is decreasing in $k$ and $2>1+2/2^2$. 
	 
	 Now (\ref{eq:5est}) holds for $k=6$, because then the left side is $2.6611\dots$ while the right side is $2.16633\dots$, and therefore it holds for all $k \geq 6$.	 
\end{proof}

\section{Monotonicity of the strip approximation entropies}\label{sec:mon}
Now we turn to the study for fixed $k$ of the strip approximation entropies $h_n^{(k)}$
{for the entropy of the golden mean $k$-tree shift (with labeling alphabet $\{0,1\}$).}
Numerical calculations indicate that these approximations {\em increase strictly} with $n$. 
We can prove this rigorously with the help of computer algebra for $k=2,\dots,8$, and we believe that it is so for all $k \geq 2$. 
The strict increase allows for another proof of Theorem \ref{th:dimincrease}, for rigorous lower bounds for $h^{(k)}$, and efficient approximations to $h^{(k)}$, in the range of $k$ for which strict increase is known to hold. 
{Rigorous upper bounds are easy because the limit in the definition of entropy is the infimum.}
\begin{theorem}\label{th:stripincrease}
	For $k=1,2,\dots ,8$ and {$n=0,1,\dots$,} we have
	\be
	\lambda_n^k < \lambda_{n+1},
	\en
	and hence 
	\be
{	h_n^{(k)} < h_{n+1}^{(k)}.}
	\en
	Moreover
	\be\label{eq:limit}
	\lim_{k \to \infty} h^{(k)} = \log 2.
	\en
\end{theorem}
\begin{remark}
	The analogue of the final statement (\ref{eq:limit}) was conjectured for the golden mean on free groups by Piantadosi \cite{Piantadosi2008} and proved already (in more generality, for $k$-tree shifts determined by arbitrary irreducible one-dimensional shifts of finite type) by the present authors in a previous paper \cite{PS2018}. 
	
	For subshifts on the lattices $\mathbb Z^d$, this ``asymptotic entropy" was known to be the increasing limit in $d$ and was proved in \cites{LouidorMarcusPavlov2013, MeyerovitchPavlov2014} to equal the ``independence entropy". 
\end{remark}	
\begin{proof}
	Fix $k=2,3,\dots$.	
	Note that {for $n \geq 0$}
	\be r_{n+1}=\frac{1}{1+r_n^k},
	\en
	so we define $T=T_k: [0,1] \to [0,1]$ by
	\be
	Tx=\frac{1}{1+x^k}, \quad 0 \leq x \leq 1.
	\en
	Then $c_{n+1}(r)=c_n(Tr)$. 
	Define also
	\be
	\begin{aligned}
		c&=c(r)=\frac{1}{2}\left[1+\sqrt{1+4r^{k-1}}\right],\\
		v&=v(r)=1+r^k.
	\end{aligned}
	\en
	Note that
	\be
	c^2-c=r^{k-1}.
	\en
		 Abbreviate {$r_{n}=r, c_n(r)=c(r)$,} and use $(Tr)^{k-1}=c(Tr)^2-c(Tr)$. 
	The statements 
	\be\label{eq:ineqa}
	c(Tr)-(Tr)^{k-1}c(r)^k>0,
	\en
	\be\label{eq:ineqb}
	n_k(r)=c^{2k}-(1+r^k)^{k-1}(c^k+1)<0,
	\en
	\be\label{eq:ineqb0}
	c^{2k}-[(c(c-1))^{k/(k-1)}+1]^{k-1}(c^k+1) <0,
	\en
	\be\label{eq:ineqc}
	\left(\frac{c+r^{k-1}}{v}\right)^k<\frac{c^k+1}{v}.
	\en
	are equivalent to one another. 
{	Moreover, when $r=r_n, n \geq -1$,} each is equivalent to
	\be\label{eq:ineqa0}
	\lambda_n^k<\lambda_{n+1}
	\en
	and implies
	\be
{\log \lambda_n < \frac{\log\lambda_{n+1}}{k}, \quad\text{ so } h_{n+1}^{(k)}=\frac{\log\lambda_{n}}{k^n} < \frac{\log\lambda_{n+1}}{k^{n+1}}=h_{n+2}^{(k)}.}
	\en

	We can prove by hand in the cases $k=2,3$ that $n_k(r) \leq 0$ for all $r \in [0,1]$,
	with equality only at the one point where $p_k(r)=r^{k+1}+r-1=0$. 
	For $k=4,5,6,7,8$, to prove this we need the help of {\em Sage} and {\em Mathematica}. We believe that given time and sufficient computer memory, (\ref{eq:ineqb}) can be verified similarly for each individual $k$. 
	So far we have not been able to prove algebraically that (\ref{eq:eqb}) has at most one solution $r \in [0,1]$ (which would suffice to prove that the statement holds for all $k$---see below).
	
	The calculation for each fixed $k$ proceeds as follows. 
	Expand $n_k(r)$ as a function of $r$, using $c(r)=(1+\sqrt{1+4r^{k-1}})/2$. 
	Group into $B_k(r)$ all the terms that contain a factor of $\sqrt{1+4r^{k-1}}$ to write $n_k(r)=A_k(r)+B_k(r)$. 
	We verify that $A_k(r)$ and $B_k(r)$ have opposite signs on $[0,1]$, take $B_k(r)$ to the right side of the formula, square both sides and take the difference. 
	Thus the left side of (\ref{eq:ineqb}) is nonpositive if and only if $d_k(r)=A_k(r)^2-B_k(r)^2 \leq 0$. 
	We define 
	\be
	q_k(r)=-\frac{d_k(r)}{p_k(r)^2} 
	\en
	and observe that $q_k(r)>0$ on $[0,1]$, concluding the argument. 
	(In some of the cases a constant is factored out, making no difference in the sign.)
	
	\begin{remark}
		It appears that $q_k(r) \to 1+2r+3r^2+\dots=1/(1-r)^2$ as $k \to \infty$ for each $r \in [0,1)$.
	\end{remark}
	
	Figure \ref{fig:diff_6,r_} is a graph of $n_6(r)$, and Figure \ref{fig:dq5} is a graph of $q_5(r)$.
	\begin{figure}
		\begin{center}
			\includegraphics[width=3.5in]{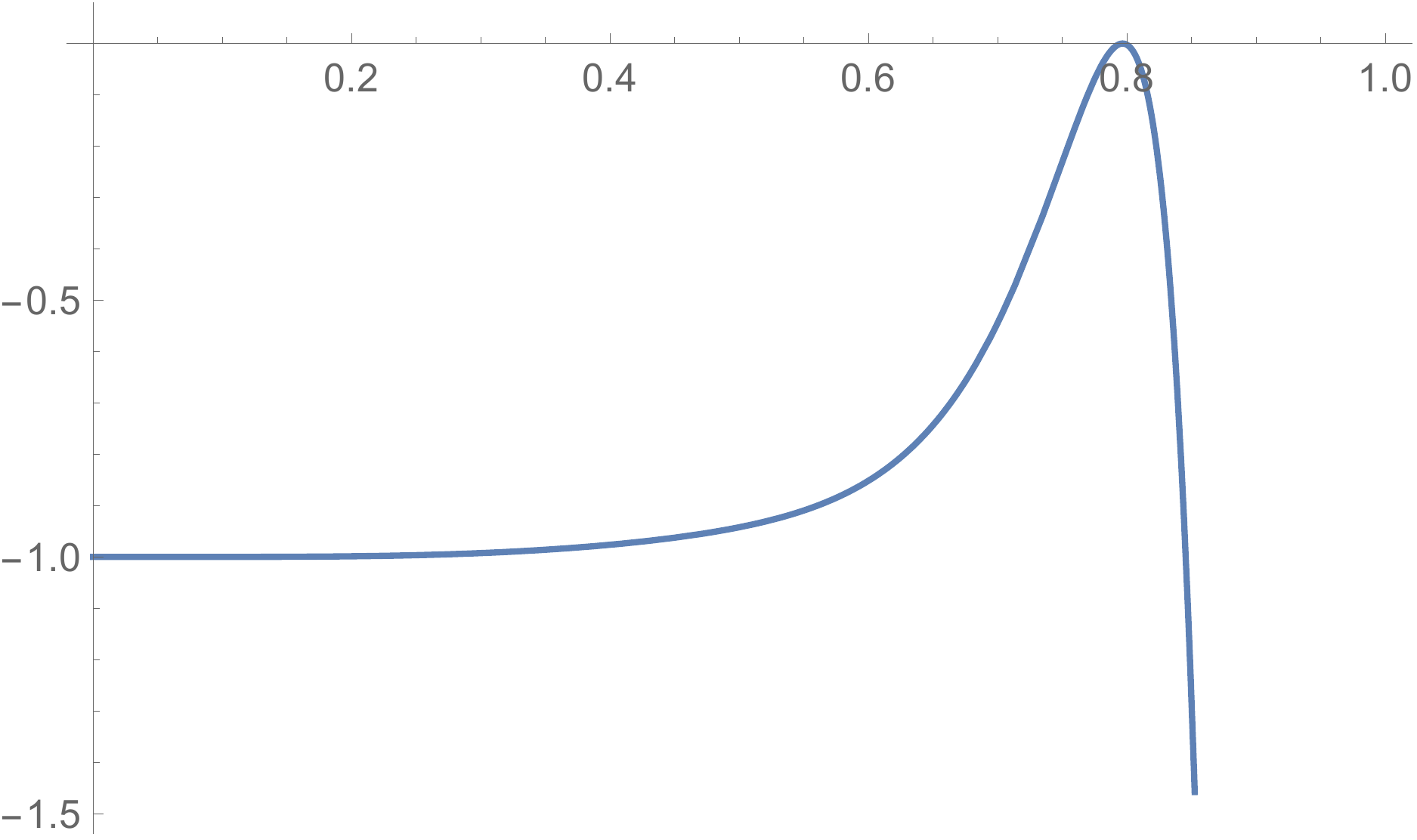}
			\caption{Graph of the difference $n_6(r)$}
			\label{fig:diff_6,r_}
		\end{center}
	\end{figure}

	\begin{figure}
		\begin{center}
			\includegraphics[width=3.5in]{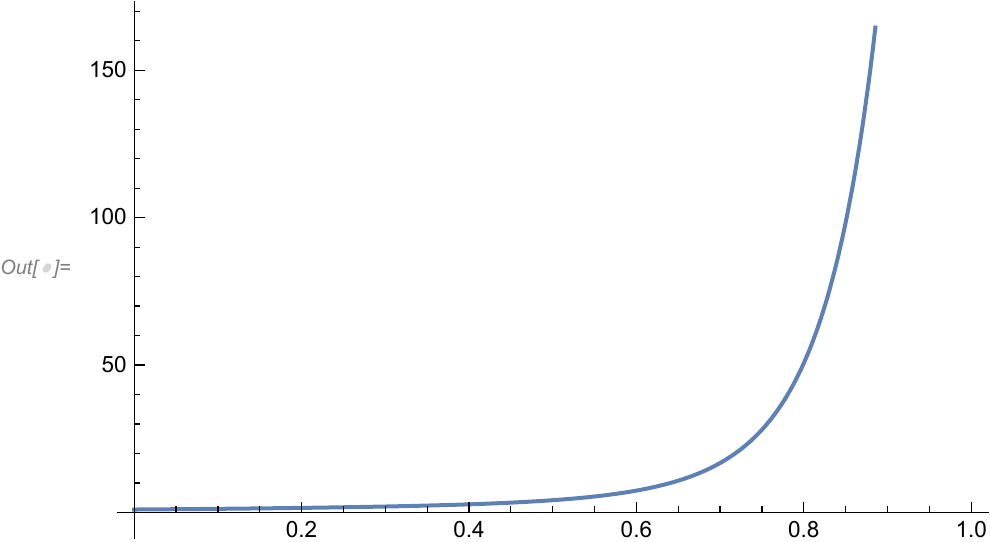}
			\caption{The nonnegative polynomial $q_k(r)$ remaining after dividing out $p_k(r)^2$ from $d_k(r)$, $k=5$}
			\label{fig:dq5}
		\end{center}
	\end{figure}

	{The results for $k=2, 3,4$ are as follows; the cases $k=5,6,7,8$ are in the Appendix. 
	Maybe contemplation of these formulas can help to find a proof for all $k$.}
	\be
	\begin{aligned}
		n_2(r)&= -1+r-\frac{r^2}{2}-r^3+r \sqrt{1+4 r}-\frac{1}{2} r^2 \sqrt{1+4 r}  \\
		A_2(r)&= -1+r-\frac{r^2}{2}-r^3   \\
		B_2(r)&= (r-r^{2}/2)\sqrt{1+4r}   \\
		d_2(r)&= 1-2 r+r^2-2 r^3+2 r^4+r^6   \\
		q_2(r)&= 1   
	\end{aligned}
	\en
	\be
	\begin{aligned}
		n_3(r)&= \frac{1}{2} (-2+3 r^2-6 r^3+9 r^4-6 r^5-r^6-3 r^8+3 r^2 \sqrt{1+4 r^2}-2 r^3 \sqrt{1+4 r^2}+\\
		&3 r^4 \sqrt{1+4 r^2}-2 r^5 \sqrt{1+4r^2}-r^6 \sqrt{1+4 r^2}-r^8 \sqrt{1+4 r^2})  \\
		A_3(r)&= -2+3 r^2-6 r^3+9 r^4-6 r^5-r^6-3 r^8   \\
		B_3(r)&= \sqrt{1+4 r^2} \left(3 r^2-2 r^3+3 r^4-2 r^5-r^6-r^8\right)   \\
		d_3(r)&= 4-12 r^2+24 r^3-36 r^4+36 r^6-72 r^7+60 r^8+8 r^9-36 r^{10}+72 r^{11}-\\
		&16 r^{12}+24 r^{14}-16 r^{15}-4 r^{18} \\  
		q_3(r)/4&= 1+2 r+4 r^3+r^4+4 r^6-2 r^7-r^{10}\\
		&= 1+2r+4r{}^{3}+r{}^{4}(1-r{}^{6})+2r{}^{6}(2-r))  
	\end{aligned}
	\en
	\be
	\begin{aligned}
		n_4(r)&= -1+2 r^3-\frac{9 r^4}{2}+9 r^6-6 r^7-\frac{9 r^8}{2}+8 r^9-3 r^{10}-6 r^{11}\\
		&-\frac{r^{12}}{2}-3 r^{14}-2 r^{15}-r^{18}+2 r^3 \sqrt{1+4
			r^3}-\frac{3}{2} r^4 \sqrt{1+4 r^3}+\\
		&5 r^6 \sqrt{1+4 r^3}-3 r^7 \sqrt{1+4 r^3}-\frac{3}{2} r^8 \sqrt{1+4 r^3}+\\
		&2 r^9 \sqrt{1+4 r^3}-3 r^{11} \sqrt{1+4	r^3}-\\
		&\frac{1}{2} r^{12} \sqrt{1+4 r^3}-r^{15} \sqrt{1+4 r^3}  \\
		A_4(r)&=  -1+2 r^3-\frac{9 r^4}{2}+9 r^6-6 r^7-\frac{9 r^8}{2}+8 r^9-3 r^{10}-6 r^{11}-\\
		&\frac{r^{12}}{2}-{3 r^{14}-2 r^{15}-r^{18}}  \\
		B_4(r)&= \sqrt{1+4 r^3} (2 r^3-\frac{3 r^4}{2}+5 r^6-3 r^7-\\
		&\frac{3 r^8}{2}+2 r^9-3 r^{11}-\frac{r^{12}}{2}-r^{15})  \\
		d_4(r)&=  1-4 r^3+9 r^4-18 r^6+27 r^8-16 r^9-48 r^{10}+36 r^{11}+37 r^{12}-48 r^{13}-\\
		&30 r^{14}+68 r^{15}+27 r^{16}-48 r^{17}+18 r^{18}+48 r^{19}+\\
		&18	r^{20}-16 r^{21}+24 r^{22}+12 r^{23}+20 r^{24}+6 r^{26}+15 r^{28}+6 r^{32}+r^{36}  \\
		q_4(r)&=  1+2 r+3 r^2+6 r^4+14 r^5+6 r^6+15 r^8+26 r^9+4 r^{10}+18 r^{12}+18 r^{13}+\\
		&4 r^{14}+9 r^{16}+6 r^{17}+6 r^{18}+2 r^{21}+4 r^{22}+r^{26}  
	\end{aligned}
	\en

\end{proof}

\begin{remark}
	We have shown that for $k=2,\dots,8$ the $h_n^{(k)}$ increase strictly to a limit, which is $h^{(k)}$, the entropy of the golden mean $k$-tree shift.
	{It is clear that $\lambda_n < \lambda_{n+1}$, but it is not clear that $h_n^{(k)} = \log \lambda_{n-1}/k^n < h_{n+1}^{(k)} = \log \lambda_{n}/k^{n+1}$.}
	The $h_n^{(k)}$ supply {\em strict lower bounds} for $h^{(k)}$, and they increase to $h^{(k)}$ very quickly, providing efficient numerical approximations. For $k$=2, we find $h^{(2)} \approx .509  > .481 \approx h^{(1)}$.
	\begin{table}[h]\label{table:approxs}
		\centering
		\begin{tabular}{l l l l l l l l l l l}
			$n$  &$B_{n-1}$  &$B_{n-1}(0)$   &$r_{n-1}(0)$ &$\log B_{n-1}/(2^n-1)$ &$\log B_{n-1}/2^n$ &$h_n^{(2)}$  \\
			\hline
			$1$  &$2$          &$1$         &$.5$        &$.693$    &$.347$    &.5025    \\  
			$2$  &$5$          &$4$         &$.8$        &$.536$    &$.402$    &.5078    \\
			$3$  &$41$         &$25$        &$.6098$     &$.531$    &$.465$    &.50866    \\
			$4$  &$2306$       &$1681$      &$.729$      &$.516$    &$.484$    &.50885    \\
			$5$  &$8,143,397$  &$5,317,636$ &$.653$      &$.513$    &$.497$    &.508889    \\
		\end{tabular}
		\medskip
		\caption{Strip entropy estimation for $k=2$}
	\end{table}
	It is curious that here we don't know that even $h_0^{(k)} \leq h^{(k)}$ without seeing first that the $h_n^{(k)}$ increase with $n$ and that their limit is $h^{(k)}$. 
	In \cite{PS2018} it is proved by linear algebra that, for $k=2$ and a tree shift with adjacency restrictions given by an irreducible $d \times d$ matrix, $h_0^{(2)}=h^{(1)}\leq h^{(2)}$. 
	But the linear algebra proof does not give the {\em strict} inequality $h^{(1)} < h^{(2)}$.
	\end{remark}

Monotonicity in $n$ of the $h_n^{(k)}$, along with knowledge that the limit in the defintion is the infimum (Theorem \ref{thm:lim=inf}), allows to show explicitly by rigorous numerical calculation that $h^{(k)}$ increases with $k$, for the values of $k$ where the monotonicity is known to hold, providing another proof of Theorem \ref{th:dimincrease} for these cases.
\begin{corollary}\label{cor:strict}
	The entropy of the hard square tree shift on the $k$-tree can {\em increase} with the dimension $k$ of the tree (in contrast with what is known about the hard square SFT on the lattice $\mathbb Z^k$ \cite{GamarnikKatz2009}). In particular, we have 
		\be
		\begin{aligned}
	h^{(1)}&\approx .481 < h^{(2)}\approx .509 < .536 < h^{(3)} <.548 \approx \frac{\log B_2^{(3)}}{1+3+3^2}\\
	&<.561 < h^{(4)} <.567 \approx \frac{\log B_2^{(4)}}{1+4+4^2}\\
	&<.58<h^{(5)}<.5839\approx \frac{\log B_2^{(5)}}{1+5+5^2}\\
	&<.5952<h^{(6)}<\dots.
	\end{aligned}
	\en
	\end{corollary}
\begin{proof}
		Given $k$, by Theorem \ref{thm:lim=inf} we can approximate $h^{(k)}$ from above by $\log B_m^{(k)}/|\Delta_m^{(k)}|$ for any $m$, and for $k=2,\dots,8$ we can approximate $h^{(k+1)}$ from below by $h_n^{(k+1)}$ for any $n$, so checking numerically that $\log B_m^{(k)}/| \Delta_m^{(k)}| < h_n^{(k+1)}$ ($n=4$ suffices for the above statements) shows rigorously that $h^{(k)}<h^{(k+1)}$ for that $k$. 
		\end{proof}

		We discuss now a plan to extend Theorem \ref{th:stripincrease} to all $k \geq 2$.
		It suffices to show that for each fixed $k \geq 2$ equality 
	in any of (\ref{eq:ineqa})--(\ref{eq:ineqc}) 
	can hold for at most one $r \in [0,1]$, which is irrational (and in fact is the solution of $Tr=r$), and so 
	in particular the inequality holds for all the $r_{n-1}, n=1,2, \dots$, which are rational.
	
	Note that as $r$ increases from $0$ to $1$, $c(r)$ increases from $1$ to the golden mean $\gamma$, $Tr$ decreases from $1$ to $1/2$, and so $c(Tr)$ decreases from $\gamma$ to $c(1/2)=(1+\sqrt{1+8/2^k})/2$, while 
	$1+1/c(r)^k$ decreases from $2$ to $1+1/\gamma^k$.. 
	
	When $r=0$, $c(Tr)=\gamma \geq1 ={c(r)^k}$, so (\ref{eq:ineqa}) holds.
	
	When $r=1$, we need to show
	\be
	1+\sqrt{1+2^{-k+3}}>4\left(\frac{\gamma}{2}\right)^k.
	\en
	Taking the $1$ from the left side to the right, squaring both sides, and simplifying leaves
	\be
	\frac{\gamma^{2k}}{2^{k-1}}< \gamma^k+1.
	\en
	Equivalently (since $\gamma^2=\gamma+1$),
	\be
	\left( \frac{\gamma+1}{2}\right)^k < \frac{\gamma^k+1}{2},
	\en
	which is easily proved by induction on $k=1,2,\dots$:
	\be 
	\left(\frac{\gamma+1}{2}\right)^{k+1} < \left(\frac{\gamma^k+1}{2}\right) \left(\frac{\gamma+1}{2}\right), \text{ which is } < \frac{\gamma^{k+1}+1}{2}
	\en
	because
	\be
	\begin{gathered}
	\gamma^{k+1}+\gamma^k+\gamma+1 < 2 \gamma^{k+1}+2 \text{ if and only if }\\
	\gamma^{k+1}-\gamma^k-\gamma+1 > 0, \text { equivalently}\\
	\gamma^k(\gamma-1)-\gamma+1 > 0, \text{ or }\\
	\gamma^k -1 > 0,
	\end{gathered}
	\en
	which is true. Thus the desired inequality holds at the extremes $r=0,1$.
	
	We want to show that there is at most one value of $r \in [0,1]$ for which equality can hold in (\ref{eq:ineqa}). Given this and the continuity of the functions being compared, the inequality \ref{eq:ineqa} holds on all of $[0,1]$ except at the point where $Tr=r$, where it is an equality.
	
	We have proved (\ref{eq:ineqb}) above by explicit algebraic calculation in the cases $k=2,3, \dots,8$, from which it emerges that its left side $n_k(r)$ is divisible by $p_k(r)^2$, where
	\be
	p_k(r)=r^{k+1}+r-1.
	\en
	
	It is difficult to make an estimate to prove any of these inequalities. For example, the graph of the left side $n_k(r)$ of inequality (\ref{eq:ineqb}) 
	touches the $r$ axis with a high degree of tangency, not allowing one to squeeze a simpler expression in between. 
	Figure \ref{fig:TwoTermsAnddiff_5,r_} shows the two terms in (\ref{eq:eqb}) of Prop. \ref{prop:eq} and their difference for $k=5$.
	\begin{figure}
		\begin{center}
			\includegraphics[width=3.5in]{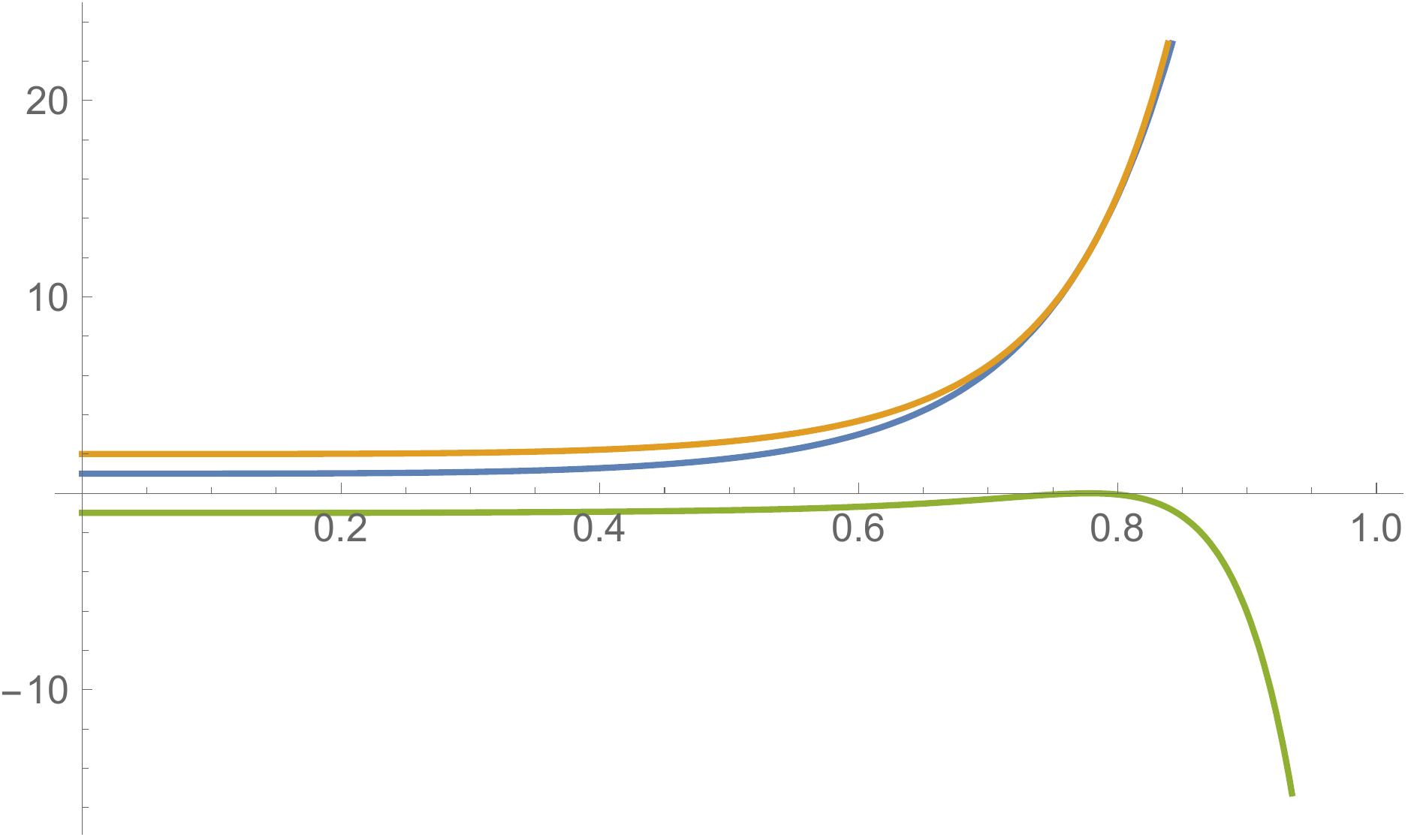}
			\caption{Graph of the two terms in (\ref{eq:eqb}) of Prop. \ref{prop:eq} and their difference, $k=5$}
			\label{fig:TwoTermsAnddiff_5,r_}
		\end{center}
	\end{figure}

	If one can obtain a workable expression for $n_k(r)/p_k(r)^2$, however, one could seek to prove, for example, 
	\be
	\frac{n_k(r)}{p_k(r)^2} \leq -\frac{1}{2} \quad\text{ for } r \in [1,\gamma],
	\en
	thus proving (\ref{eq:ineqb}).

	We believe that the irrationality of the single solution will follow from Gauss' Lemma.

The following observations may be helpful in completing the proof of strict increase of $h_n^{(k)}$ for each $k$.

\begin{proposition}\label{prop:eq}
	The following are equivalent:
	\be\label{eq:eq1}
	Tr=r,
	\en
	\be\label{eq:eq2}
	r^{k+1}+r-1=0,
	\en
	\be\label{eq:eq3}
	cr=1,
	\en
	\be\label{eq:eq4}
	c^{k+1}-c^k-1=0.
	\en
	Moreover, each of (\ref{eq:eq1})--(\ref{eq:eq4}) implies
	\be\label{eq:eqa0}
	\lambda_n^k=\lambda_{n+1},
	\en
	\be\label{eq:eqa}
	c(Tr)-(Tr)^{k-1}c(r)^k=0,
	\en
	\be\label{eq:eqb}
	c^{2k}-(1+r^k)^{k-1}(c^k+1)=0,
	\en
	\be\label{eq:eqb0}
	c^{2k}-[(c(c-1))^{k/(k-1)}+1]^{k-1}(c^k+1) =0,
	\en
	\be\label{eq:eqc}
	\left(\frac{c+r^{k-1}}{v}\right)^k=\frac{c^k+1}{v}.
    \en
	\end{proposition}.
\begin{proof}
	That (\ref{eq:eq1}) is equivalent to (\ref{eq:eq2}) is clear.
	
	(\ref{eq:eq1}) implies (\ref{eq:eq3}) because $Tr=r$ implies 
	\be\frac{1}{r^2}-\frac{1}{r}=r^{k-1}=c^2-c,
	\en
	so that
	\be
	0=c^2-\frac{1}{r^2}-(c-\frac{1}{r})=(c-\frac{1}{r})[(c+\frac{1}{r})-1],
	\en
	which implies that $c=1/r$.
	
	(\ref{eq:eq3}) implies (\ref{eq:eq1}) because if $c=1/r$ we have
	\be
	r^{k-1}=c^2-c=\frac{1}{r^2}-\frac{1}{r},
	\en
	so that Tr=r.
	
	To see that (\ref{eq:eq3}) implies (\ref{eq:eq4}), note that we already know that (\ref{eq:eq2}) is equivalent to (\ref{eq:eq3}), so that $cr=1$ implies
	\be
	c^{k+1}-c^k-1=\frac{1}{r^{k+1}}-\frac{1}{r^k}-1=r^{k+1}(1-r-r^{k+1})=0.
	\en
	
	Finally, (\ref{eq:eq4}) implies (\ref{eq:eq3}) because
	\be
	0=c^{k+1}-c^k-1=c^{k-1}(c^2-c)-1=c^{k-1}r^{k-1}-1
	\en
	implies $cr=1$.

Turning now to the second part of the Proposition, if $Tr=r$ then (\ref{eq:eqa}) says $c(r)=r^{k-1}c(r)^k$, which is true because $cr=1$. 

We finish the argument by showing that (\ref{eq:eqa0})--(\ref{eq:eqc}) are equivalent to one another. 
Note that $\lambda_n^k=\lambda_{n+1}$ if and only if
\be
f_{n+1}(\lambda_n^k)=\lambda_n^{2k}-B_{n+1}^{k-1}\lambda_n^k-[B_{n+1}B_{n+1}(0)]^{k-1}=0,
\en
equivalently, since $c_n=\lambda_n/B_m^{k-1}$,
\be\label{eq:eqB}
\frac{f_{n+1}(\lambda_n^k)}{B_n^{2k(k-1)}}=c_n^{2k}-c_n^k(1+r_n^k)^{k-1}-(1+r_n^k)^{k-1}=0.
\en
This shows that (\ref{eq:eqa0}) is equivalent to (\ref{eq:eqb}).

Further, $\lambda_n^k=\lambda_{n+1}$ says that
\be
B_n^{k(k-1)}c_n^k=B_{n+1}^{k-1}c_{n+1}=[B_n^k(1+r_n^k)]^{k-1}c_{n+1},
\en
equivalently
\be
c_n^k=(1+r_n^k)^{k-1}c_{n+1}.
\en
This shows that (\ref{eq:eqa0}) is equivalent to (\ref{eq:eqa}).

	That (\ref{eq:eqb}) is equivalent to each of (\ref{eq:eqb0}) and (\ref{eq:eqc}) is direct from $c^2-c=r^{k-1}$.
\end{proof}

\section{Map of the simplex}
{Consider a $k$-tree SFT $Z_M$ with 
 $k \geq 2$, alphabet size $d=2$, and transition matrix $M=[11,10]$, so that $Z_M$ consists of all labelings of the $k$-tree with no adjacent nodes both labeled by $1$.}
Consider first the case $k=2$. 
In this case $T:K_1 \to K_1$ is a function of a single variable: $T(x, 1-x)=(1, x^2)/(1+x^2)$.
The equations $Tx=x$ reduce to $x^3+x-1=0$, with root $u=0.682328$. 
The derivative at the fixed point is $T'(u)=-0.635345$.
The fixed point is attracting: $T^nx \to u$ for all $x \in [0,1]$.

In the case of the $k$-tree, $k \geq 2$, $T_k:K_1 \to K_1$ (again identifying $K_1$ with $[0,1]$ according to $(x,1-x) \leftrightarrow x$) 
is the function $T_kx=1/(1+x^k)$. We may consider this function for all $k>0$, not necessarily an integer. 
When $k$ is fixed, we will suppress the subscript on $T_k$.
Again there is a fixed point $u_k=u \in [0,1]$, which is unique because $T$ is decreasing. 
For small $k$ we have $|T'(u)|<1$, so the fixed point is attracting. 
But there is a critical value $k_0 \approx 4.125$, the solution of
\be
k_0=1 + k_0^{k_0/(k_0+1)},
\en
at which $T'(u)=1$. 
We claim that for $k>k_0$ besides the fixed point $u$ there is also an attracting periodic orbit $\{p_1,p_2\}$, and there are no other periodic points.
Figure \ref{fig:f7} shows the graphs of $T, T^2$, and the identity function for $k=7$.
\begin{figure}
	\begin{center}
		\includegraphics[width=3.5in]{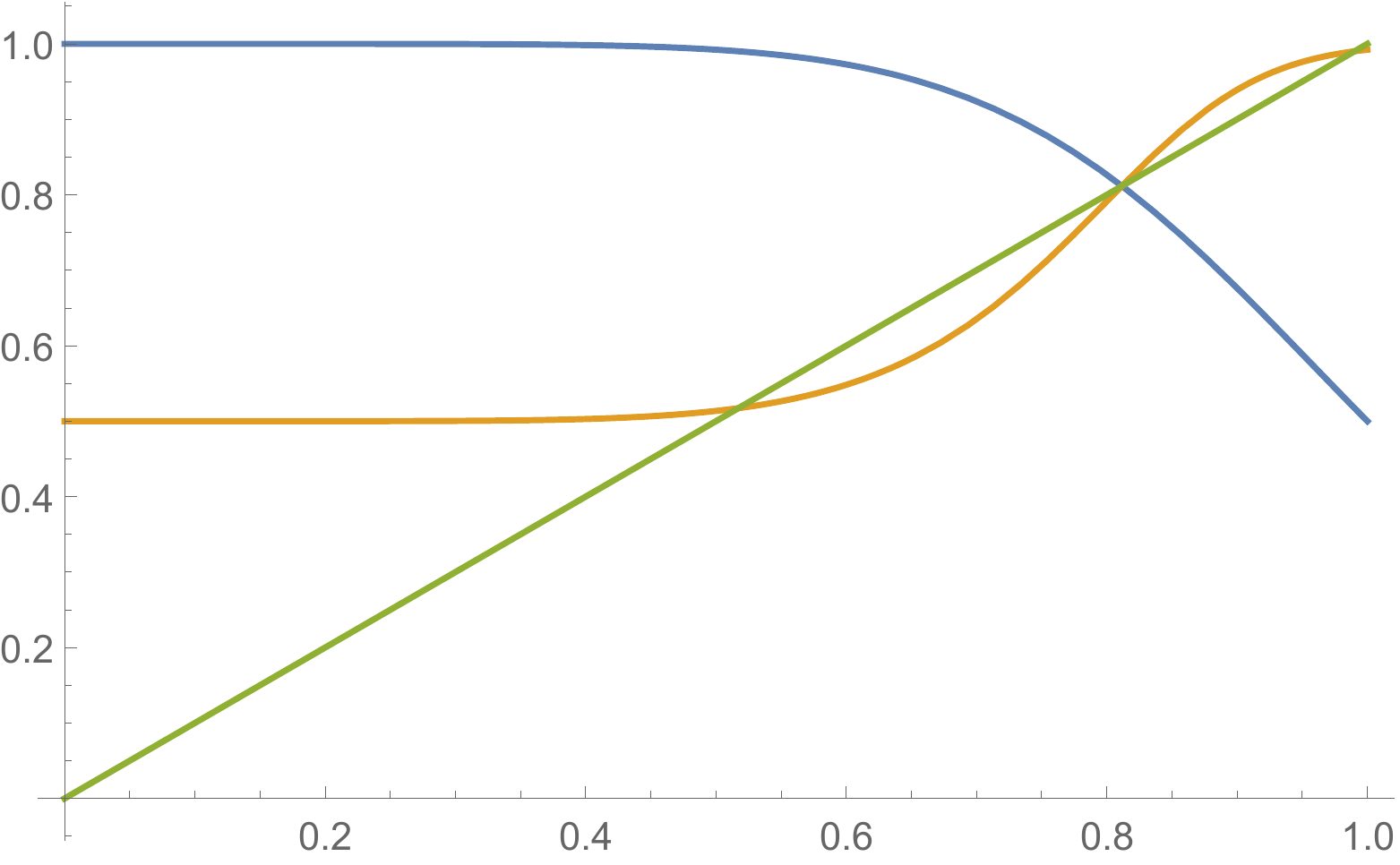}
		\caption{The interval map and its square for $k=7$}
		\label{fig:f7}
	\end{center}
\end{figure}

\begin{theorem}\label{thm:intervalmap}
	(i) For $k<k_0$ the map $T=T_k$ on $K_1$ has a fixed point $u=u_k$ which attracts all of $[0,1]$, and there are no other periodic orbits.\\
	(ii) For $k>k_0$ the map $T=T_k$ on $K_1$ has a repelling fixed point $u=u_k$, an attracting periodic orbit $\{p_1,p_2\}$ which attracts all of $[0,1]\setminus \{ u \}$, and no other periodic orbits.
\end{theorem}
\begin{proof}
	Since $T$ is decreasing and $T^2$ is increasing, we have $0 \to T0=1 \to T1= 1/2 \to \dots$, so that $T^{2n}0 \nearrow p_1$ and $T^{2n+1} 0 \searrow p_2$, with $p_1 \leq u \leq p_2$. 
	
	Moreover, $T[0,p_1]=[p_2,1]$, and $T[p_2,1]=[1/2,p_1]$. 
	Similarly, $T[p_1,u]=[u,p_2], T[u,p_2]=[p_1,u]$.
	
	Computer assisted computation shows that if $k<k_0$ then
	\be
	{|DT^2(x)| \leq c_k <1 \quad\text{for all }  x \in [0,1].}
	\en
	Thus for $k<k_0$, $p_1=u=p_2$ and all of $[0,1]$ is attracted to the unique fixed point $u$. 
	
	We show now that on the other hand, for $k>k_0$ the periodic cycle $\{p_1,p_2\}$ attracts all of $[0,1]$. 
	First, $[0,p_1] \cup [p_2,1]$ is attracted to $\{p_1,p_2\}$, because of the observations above. 
	
	Note further that if $k>k_0$, then at the fixed point $u$ of $Tx=1/(1+x^k)$ we have $|T'(u)|>1$. 
	Because consider the function $g(k)=k^{k/(k+1)}+1$ for $k \geq 0$.  
	We have $g(0)=2, g(k_0)=k_0, g'(k)<0$ for small $k$, $g'(k)=0$ when $\log k =k+1$ ($k=k_1 \approx 0.3$), $g'(k)>0$ for $k>k_1$. 
	Thus $g(k)>k$ for small $k$, there is a unique solution $k_0\approx 4.125$ of $g(k)=k$, and $g(5) \approx 4.82362<5$, so $g(k)<k$ for $k>k_0$. See Figure \ref{fig:g(k)}.
	\begin{figure}
		\begin{center}
			\includegraphics[width=3.5in]{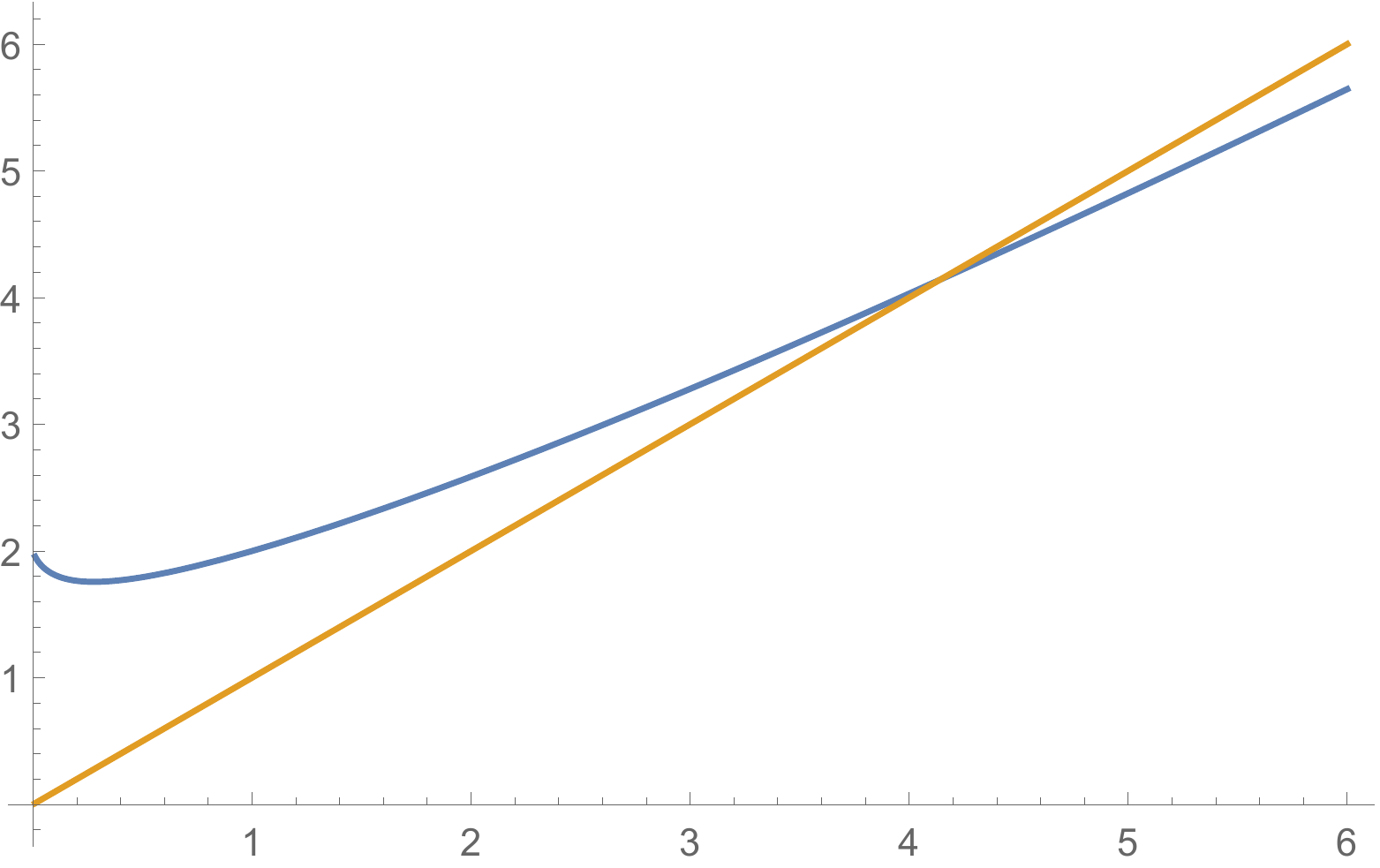}
			\caption{The function $g(k)$}
			\label{fig:g(k)}
		\end{center}
	\end{figure}
	
	This implies that for $k>k_0$ we have $|T_k'(u)|>1$.  
	Because for $k>k_0$ we have
	\be
	\begin{gathered}
	1+k^{k/(k+1)}<k, \quad k-1>k^{\frac{k}{k+1}}, \quad(k-1)^{k+1}>k^k,\\
	\frac{(k-1)^{k+1}}{k^{k+1}}>\frac{1}{k}, \quad 1<1-\frac{1}{k}+(1-\frac{1}{k})^{k+1}=\frac{k-1}{k}+\left(\frac{k-1}{k}\right)^{k+1},\\
	\frac{1}{1+\left(\frac{k-1}{k}\right)^k}<\frac{k-1}{k};
	\end{gathered}
	\en
	which is to say that
	\be
	T_k\left(\frac{k-1}{k}\right) < \frac{k-1}{k},
	\en
	so that $(k-1)/k>u, k-ku>1$ for $k>k_0$. 
	But
	\be
	T_k'(u)=\frac{-ku^{k-1}}{(1+u^k)^2}=-ku^{k+1}
	\en
	since $u^{k+1}=1-u$.
	Therefore 
	\be
	|T_k'(u)|=ku^{k+1}=k-ku>1.
	\en

	Since $|T'(u)|>1$, there are $\delta, \epsilon>0$ such that $|Tx-u| \geq (1+\delta)|x-u|$ for $|x-u|\leq \epsilon$. 
	If we take a point $x>u$ close to $u$, we have $Tx<u, T^2x>u$. 
	We must have $u<x<T^2x$, because otherwise $u<T^2x<x$ and
	\be
	|T^2x-u| \geq (1+\delta)|Tx-u| \geq (1+\delta)^2|x-u| \geq |x-u|,
	\en
	a contradiction. 
	Thus $T^2$ moves points near $u$ away from $u$, towards either $p_1$ or $p_2$, depending on whether $x \in (p_1,u)$ or $x \in (u,p_2)$.
	
	We show now that the $2$-cycle $\{ p_1, p_2\}$ attracts all of $[0,1] \setminus \{ u \}$. 
	This will follow if we can show that (for $k > k_0$) $T^2$ has a unique inflection point.
	Because from above, we have that for $p_1 < x < u$ and $x$ near $u$, $p_1 < T^2x < x < u$.
	Since $T^2$ is increasing, even though maybe no longer is $T^2x \approx u$, we may iterate to find 
	\be
	p_1 < T^4x < T^2x <x <u,
	\en
	so that $T^{2n}$ decreases to a fixed point $p \geq p_1$ of $T^2$. 
	
	But if $T^2$ has a unique inflection point in $[0,1]$, we must have $p=p_1$. 
	This is because then the graph of $T^2$ hits the line $y=x$ at least at $p_1, p, u, Tp, p_2$, and so must change concavity at  least twice. 
	This shows that all of $(p_1,u)$ is attracted to $p_1$ under $T^2$, and therefore all of $[0,1] \setminus \{u \}$ is attracted to $\{p_1, p_2\}$ under $T$.
	
	We show now that $T^2$ has a unique inflection point in $[0,1]$. 
	With the help of {\em Sage}, we find that  
	\be
	T^2=\frac{(x^k + 1)^k}{(1+(x^k + 1))^k},
	\en
	\be
	(T^2)'=\frac{k^2 (\frac{x}{x^k+1})^{k-1}}{(x^k+1)^2(\frac{1}{(x^k+1)^k}+1)^2},
	\en
	\be
	(T^2)''=\frac{(k^2(\frac{x}{(x^k + 1)})^k - k^2x^k - (\frac{x}{(x^k + 1)}))^k + \frac{k}{(x^k + 1)^k} + k - x^k - 
		\frac{1}{(x^k + 1)^k} - 1)k^2\frac{x^{k - 2}}{(x^k + 1)^k}}{(x^k + 1)^2(\frac{1}{(x^k + 1)^k} + 1)^3}.
	\en
	Let $y=x^k$ and
	\be
	n1(x,k)=(T^2)''((x^k + 1)^2((1/(x^k + 1)^k + 1)^3),
	\en
	\be
	n2(x,k)=k^2(\frac{x}{(x^k + 1)})^k - k^2x^k - (\frac{x}{(x^k + 1)})^k + \frac{k}{(x^k + 1)^k} + k - x^k - 
	\frac{1}{(x^k + 1)^k} - 1,
	\en
	\be
	n3(x,k)=n2(x,k)(x^k+1)^k,
	\en
	\be
	n4(x,k)=k^2x^k(1-(x^k+1)^k)-x^k+k+(k-x^k)(x^k+1)^k-1-(x^k+1)^k, \quad\text{ and }
	\en
	\be
	j(y,k)=k^2y(1-(y+1)^k)-y+k+(k-y)(y+1)^k-1-(y+1)^k,  
	\en
	all of which have the same sign as $(T^2)''$ on $[0,1]$.
	
	Now $j(y,k)$ is a polynomial in $y$ of degree $k+1$. The constant term and coefficient of $y$ are positive, and 
	for $j>1$ the coefficient of $y^j$ is
	\be
	(k-1) \binom{k}{j} - (k^2+1) \binom{k}{j-1},
	\en
	which is negative for $k>2$. 
	So the second derivative of $j(y,k)$ is negative, its first derivative starts positive at $0$ and later is negative so that $j$ has a unique zero on the positive axis.
	Since $(T^2)''$ has the same sign as $j$, $T$ has a unique inflection point. 
	Therefore, from above, for $k>k_0$ $T$ has a unique fixed point and an attracting cycle of period $2$ which attracts everything else, and there are no other periodic points. 
\end{proof}

\section{Intermediate entropy}
{For a tree labeled by elements of a finite alphabet $A=\{0,1,\dots ,d-1\}$, 
one can count numbers of possible labelings not just for the fundamental height $n$ subtrees but for any finite patterns. 
Focusing in this section on the binary regular tree, 
we are interested in particular patterns which consist of a translate of $\Delta_n$ and an `initial' segment of the next row (the final row of the corresponding translate of $\Delta_{n+1}$. 
Let us label the nodes $\eta$ left to right in each row, continuing downward from row to row: so $\epsilon \leftrightarrow \eta(1), 0\leftrightarrow\eta(2), 1\leftrightarrow\eta(3), 00\leftrightarrow\eta(4), 01\leftrightarrow\eta(5), 10\leftrightarrow\eta(6), 11\leftrightarrow\eta(7), 000\leftrightarrow\eta(8), ...$.  
(Recall that nodes of the tree correspond to words in $\{0,1\}^*$---see Section \ref{sec:setup}). }
Define $q(n)$ to be the number of labelings of $\eta(1) \eta(2) \dots \eta(n)$ in the tree shift. 
(For a single labeled tree, we define $q(n)$ to be the number of labelings found among all translates of $\eta(1) \eta(2) \dots \eta(n)$.)
For $n=0,1,\dots$ let $c_n=|\Delta_n|=2^{n+1}-1$. Thus
\be
q(c_n)=p(n) \quad\text{for all } n=0,1,\dots .
\en
We ask whether the {\em intermediate entropy} $h^i=\lim_{n \to \infty} \log q(n)/n$ exists (and hence equals the entropy $h=\lim_{n \to \infty} \log q(c_n)/c_n$ of the tree shift or labeled tree). 
For now let us define the {\em upper and lower intermediate entropies} of a tree shift or labeled tree to be 
\be
h^u= \limsup_{n \to \infty} \frac{\log q(n)}{n} \quad\text{ and } h^l= \liminf_{n \to \infty} \frac{\log q(n)}{n}.
\en
\begin{proposition}
	{The upper intermediate entropy of a $2$-tree shift or labeled $2$-tree cannot exceed its entropy: $h^u \leq h$.}
\end{proposition}
\begin{proof}
	Fix $n$ and $j=0, \dots, 2^{n+1}$ and consider the pattern $\Delta_n(j)$ that consists of $\Delta_n$ together with an initial segment of length $j$ of the next row in $\Delta_{n+1}$. 
	We wish to estimate from above the number of possible labelings of  $\Delta_n(j)$ that can be found in the tree shift. 
	Consider first the case when $j=2^n$, i.e. we use half of the next row; see Figure \ref{fig:halfway}.  
	
	\
	
	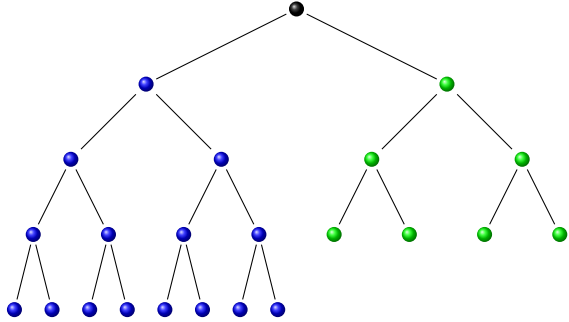
\begin{figure}[h]
		\begin{tikzpicture}
		\node at (0,0)  (v) {};
		\shade [ball color=black] (v) circle (0.1);
		\node at (-2,-1) (v0) {};
		\shade [ball color=blue] (v0) circle (0.1);
		\node at (2,-1)  (v1) {};
		\shade [ball color=green] (v1) circle (0.1);
		\node at (-3,-2) (v00) {};
		\shade [ball color=blue] (v00) circle (0.1);
		\node at (-1,-2) (v01) {};
		\shade [ball color=blue] (v01) circle (0.1);
		\node at (1,-2) (v10) {};
		\shade [ball color=green] (v10) circle (0.1);
		\node at (3,-2) (v11) {};
		\shade [ball color=green] (v11) circle (0.1);
		\node at (-3.5,-3) (v000) {};
		\shade [ball color=blue] (v000) circle (0.1);
		\node at (-2.5,-3) (v001) {};
		\shade [ball color=blue] (v001) circle (0.1);
		\node at (-1.5,-3) (v010) {};
		\shade [ball color=blue] (v010) circle (0.1);
		\node at (-.5,-3) (v011) {};
		\shade [ball color=blue] (v011) circle (0.1);
		\node at (.5,-3) (v100) {};
		\shade [ball color=green] (v100) circle (0.1);
		\node at (1.5,-3) (v101) {};
		\shade [ball color=green] (v101) circle (0.1);
		\node at (2.5,-3) (v110) {};
		\shade [ball color=green] (v110) circle (0.1);
		\node at (3.5,-3) (v111) {};
		\shade [ball color=green] (v111) circle (0.1);
		\node at (-3.75,-4) (v0000) {};
		\shade [ball color=blue] (v0000) circle (0.1);
		\node at (-3.25,-4) (v0001) {};
		\shade [ball color=blue] (v0001) circle (0.1);
		\node at (-2.75,-4) (v0010) {};
		\shade [ball color=blue] (v0010) circle (0.1);
		\node at (-2.25,-4) (v0011) {};
		\shade [ball color=blue] (v0011) circle (0.1);
		\node at (-1.75,-4) (v0100) {};
		\shade [ball color=blue] (v0100) circle (0.1);
		\node at (-1.25,-4) (v0101) {};
		\shade [ball color=blue] (v0101) circle (0.1);
		\node at (-.75,-4) (v0110) {};
		\shade [ball color=blue] (v0110) circle (0.1);
		\node at (-.25,-4) (v0111) {};
		\shade [ball color=blue] (v0111) circle (0.1);
		
		\draw (v) -- (v0);
		\draw (v) -- (v1);
		\draw (v0) -- (v00);
		\draw (v0) -- (v01);
		\draw (v1) -- (v10);
		\draw (v1) -- (v11);
		\draw (v00) -- (v000);
		\draw (v00) -- (v001);
		\draw (v01) -- (v010);
		\draw (v01) -- (v011);
		\draw (v10) -- (v100);
		\draw (v10) -- (v101);
		\draw (v11) -- (v110);
		\draw (v11) -- (v111);
		\draw (v000) -- (v0000);
		\draw (v000) -- (v0001);
		\draw (v001) -- (v0010);
		\draw (v001) -- (v0011);
		\draw (v010) -- (v0100);
		\draw (v010) -- (v0101);
		\draw (v011) -- (v0110);
		\draw (v011) -- (v0111);
		\end{tikzpicture}
		\caption{Including half of the next row.}
		\label{fig:halfway}
	\end{figure}

	$\Delta_n(2^n)$ consists of the root, a translate of $\Delta_{n}$ (the blue nodes in Figure \ref{fig:halfway}), and a translate of $\Delta_{n-1}$ (the green nodes in Figure \ref{fig:halfway}). 
	There are $d$ choices for labels of the root, $q(c_{n})$ choices for labelings of $\Delta_{n}$, and $q(c_{n-1})$ choices for labelings of $\Delta_{n-1}$, so
	\be
	q(c_n+2^n) \leq d q(c_n) q(c_{n-1}) .
	\en
	Since $c_n + c_{n-1}=c_n+2^n-1$,
	\be
	\lim_{n \to  \infty} \frac{\log q(c_n)}{c_n} = h,
	\en
	and
	\be
	\frac{\log q(c_n+2^n)}{c_n+2^n} \leq \frac{\log d}{c_n+2^n} + \frac{c_n}{c_n+2^n} \frac{\log q(c_n)}{c_n} + \frac{c_{n-1}}{c_n+2^n}\frac{\log q(c_{n-1})}{c_{n-1}},
	\en
	we have
	\be
	\limsup_{n \to \infty} \frac{\log q(c_n+2^n)}{c_n+2^n} \leq h.
	\en
	
	Let $\epsilon >0$, choose $r$ so that $1/2^r<\epsilon$, and suppose that $n>>r$..
	We divide the $2^{n+1}$ nodes on the last row of $\Delta_{n+1}$ into $2^r$ consecutive intervals, each of length $2^m$, where $m={n-r+1}$. 
	Suppose first that $j=k2^m$ for some $k=0,\dots, 2^r-1$.
	Write $k=k_0+k_12+\dots +k_r2^{r-1}$ with each $j_i=0$ or $1$ and let $i_0=\inf\{i: k_i \neq 0\}$.
	We proceed to decompose $\Delta_n(j)$ as a subset of the disjoint union of the root and at most $n$ other nodes, a translate of $\Delta_{i_0}(2^{i_0})$, and translates of $\Delta_{i_0}, \Delta_{i_0+1}, \dots, \Delta_{n}$. 
	Note that $|\Delta_{i_0}(2^{i_0})|+|\Delta_{i_0}| + |\Delta_{i_0+1}| + \dots + |\Delta_{n}|$ is within $n$ of $|\Delta_n(j)|$.
	Figure \ref{fig:5/8} shows an example with $n=3,r=3,j=10,r=3,m=1,k=5=1+0\cdot 2+1\cdot2^2, i_0=0$. 
	Here the green nodes form a translate of $\Delta_{i_0}(2^{i_0})$ (a subset of a translate of $\Delta_2$), the red nodes form a translate of $\Delta_{i_0}$, the blue nodes form a translate of $\Delta_n$, and  the black nodes are the ``free'' ones. 
	
	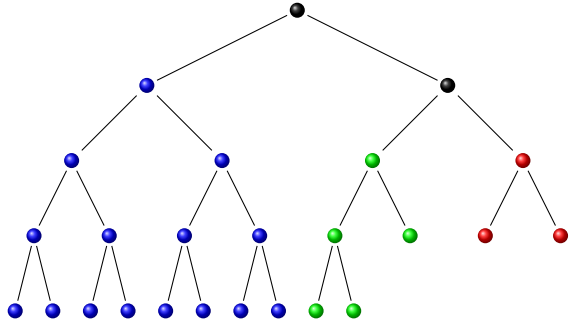
\begin{figure}[h]
		\begin{tikzpicture}
		\node at (0,0)  (v) {};
		\shade [ball color=black] (v) circle (0.1);
		\node at (-2,-1) (v0) {};
		\shade [ball color=blue] (v0) circle (0.1);
		\node at (2,-1)  (v1) {};
		\shade [ball color=black] (v1) circle (0.1);
		\node at (-3,-2) (v00) {};
		\shade [ball color=blue] (v00) circle (0.1);
		\node at (-1,-2) (v01) {};
		\shade [ball color=blue] (v01) circle (0.1);
		\node at (1,-2) (v10) {};
		\shade [ball color=green] (v10) circle (0.1);
		\node at (3,-2) (v11) {};
		\shade [ball color=red] (v11) circle (0.1);
		\node at (-3.5,-3) (v000) {};
		\shade [ball color=blue] (v000) circle (0.1);
		\node at (-2.5,-3) (v001) {};
		\shade [ball color=blue] (v001) circle (0.1);
		\node at (-1.5,-3) (v010) {};
		\shade [ball color=blue] (v010) circle (0.1);
		\node at (-.5,-3) (v011) {};
		\shade [ball color=blue] (v011) circle (0.1);
		\node at (.5,-3) (v100) {};
		\shade [ball color=green] (v100) circle (0.1);
		\node at (1.5,-3) (v101) {};
		\shade [ball color=green] (v101) circle (0.1);
		\node at (2.5,-3) (v110) {};
		\shade [ball color=red] (v110) circle (0.1);
		\node at (3.5,-3) (v111) {};
		\shade [ball color=red] (v111) circle (0.1);
		\node at (-3.75,-4) (v0000) {};
		\shade [ball color=blue] (v0000) circle (0.1);
		\node at (-3.25,-4) (v0001) {};
		\shade [ball color=blue] (v0001) circle (0.1);
		\node at (-2.75,-4) (v0010) {};
		\shade [ball color=blue] (v0010) circle (0.1);
		\node at (-2.25,-4) (v0011) {};
		\shade [ball color=blue] (v0011) circle (0.1);
		\node at (-1.75,-4) (v0100) {};
		\shade [ball color=blue] (v0100) circle (0.1);
		\node at (-1.25,-4) (v0101) {};
		\shade [ball color=blue] (v0101) circle (0.1);
		\node at (-.75,-4) (v0110) {};
		\shade [ball color=blue] (v0110) circle (0.1);
		\node at (-.25,-4) (v0111) {};
		\shade [ball color=blue] (v0111) circle (0.1);
		
		\node at (.25,-4) (v1000) {};
		\shade [ball color=green] (v1000) circle (0.1);
		\node at (.75,-4) (v1001) {};
		\shade [ball color=green] (v1001) circle (0.1);
		
		\draw (v) -- (v0);
		\draw (v) -- (v1);
		\draw (v0) -- (v00);
		\draw (v0) -- (v01);
		\draw (v1) -- (v10);
		\draw (v1) -- (v11);
		\draw (v00) -- (v000);
		\draw (v00) -- (v001);
		\draw (v01) -- (v010);
		\draw (v01) -- (v011);
		\draw (v10) -- (v100);
		\draw (v10) -- (v101);
		\draw (v11) -- (v110);
		\draw (v11) -- (v111);
		\draw (v000) -- (v0000);
		\draw (v000) -- (v0001);
		\draw (v001) -- (v0010);
		\draw (v001) -- (v0011);
		\draw (v010) -- (v0100);
		\draw (v010) -- (v0101);
		\draw (v011) -- (v0110);
		\draw (v011) -- (v0111);
		
		\draw (v100) -- (v1000);
		\draw (v100) -- (v1001);
		\end{tikzpicture}
		\caption{Using 5/8 of the next row.}
		\label{fig:5/8}
	\end{figure}
	
	This decomposition yields the estimate
	\be
	q(c_n+j) \leq d^n q(c_{i_0}) q(c_{i_0-1}) \prod_{t=i_0}^{n-1} q(c_t).
	\en
	This implies that with $r$ fixed, once $n$ is large enough that
	\be
	\frac{\log(q(c_t))}{c_t} < h+ \epsilon \quad \text{ for most } t=i_0, \dots, n-1,
	\en 
	we will have for all $j=k2^m$ for some $k=0,\dots, 2^r-1$,
	\be
	\frac{\log q(c_n+j)}{c_n+j} < h + 2\epsilon.
	\en
	
	Continuing with $r$ fixed, we consider now values of $j$ which are not multiples of $2^m$. 
	For any such $j=1, \dots, 2^{n+1}-1$, let $j'=\inf\{k2^m>j\}$. Then
	\be
	\begin{aligned}
		\frac{\log q(c_n+j)}{c_n+j} &\leq \frac{\log q(c_n+j')}{c_n+j'}\frac{c_n+j'}{c_n+j}
		< (h+\epsilon)  \frac{c_n+j+2^m}{c_n+j}\\
		&= (h+\epsilon)(1 +\frac{2^m}{c_n}) < h+3 \epsilon 
	\end{aligned}
	\en
	for large enough $n$. 
\end{proof}

\begin{corollary}\label{cor:intent}
	Let $Z=Z(X)$ be the 2-tree shift corresponding to a 1-dimensional shift of finite type $X$. Then the intermediate topological entropy of $Z(X)$ equals its topological entropy: $h^l(Z(X)) \geq h(Z(X))\geq h^u(Z(X)) \geq h(X)$.
	More generally, if $Z$ is a tree SFT (determined by excluding a finite set of blocks), then its intermediate topological entropy equals its topological entropy: $h^i(Z)=h(Z)$. 
\end{corollary}
\begin{proof}
	In a tree shift corresponding to a one-step one-dimensional shift of finite type, labelings of disjoint translates of basic patterns $\Delta_n$ are independent given the configurations above them, so the key inequalities in the above proof are actually equalities. For systems with more memory (still bounded), consider the appropriate higher block coding.
	
	For the general case of a tree SFT, labelings of sets of nodes in distinct branches of the tree are again independent, given a fixed allowed configuration above both of them: pairing two such allowed labelings produces an allowed labeling of their union. 	
\end{proof}

\section{Larger alphabets and more general labeling restrictions}\label{sec:general}
 Let $M$ be an irreducible $d \times d$ $\{0,1\}$ matrix and $Z_M$ the $k$-tree shift consisting of all $k$-trees labeled by elements of the alphabet $A=\{0, \dots, d-1\}$ consistently with the transitions allowed by $M$, as in Section \ref{sec:setup}. 
 We describe briefly the setup for applying the strip method to this general situation to study the topological entropy $h(Z_M)$ of the system $Z_M$. 
 
 For each $n=0,1,\dots$ let $x(n)=(x_0(n),\dots,x_{d-1}(n))$ denote the vector of symbol counts in which $x_i(n)$ is the number of labeled translates of $\Delta_n$ ($n$-blocks) with the symbol $i$ at the root. 
 Let $|x(n)|=x_0(n)+\dots+x_{d-1}(n)$
 As before, denoting by $M_i$ the $i$'th row of $M$, we have the recursion
 \be
 x_i(n+1)=(Mx(n))_i^k=[M_i \cdot x(n)]^k, \quad i=0,\dots,d-1.
 \en
Define
\be
g_k(x)= \sum_j(Mx)_j^k
\en
and the map $T_k: \mathbb R^d \to \mathbb R^d$ by
\be
(T_kx)_i= \frac{(Mx)_i^k}{\sum_j(Mx)_j^k}=\frac{(Mx)_i^k}{g_k(x)}, \quad i=0,1,\dots, d-1,
\en
and as before
\be
\begin{aligned}
	r(n)&=\frac{x(n)}{|x(n)|}, \quad\text{so that}\\
	r(n+1)&=T_kr(n).
\end{aligned}
\en

For the $1$-dimensional golden mean SFT, the map $T$ on $K_1=\{(x,y):x\geq 0, y\geq 0, x+y=1\}$ is
$T(x,y)=( (x+y)^2, x^2)/((x+y)^2+x^2)$. 
The equations $T(x,y)=(x,y)$ define an algebraic curve, part of which is shown in Figure \ref{fig:GMCurve} along with the simplex $x+y=1$. 
\begin{figure}
	\begin{center}
		\includegraphics[width=3.5in]{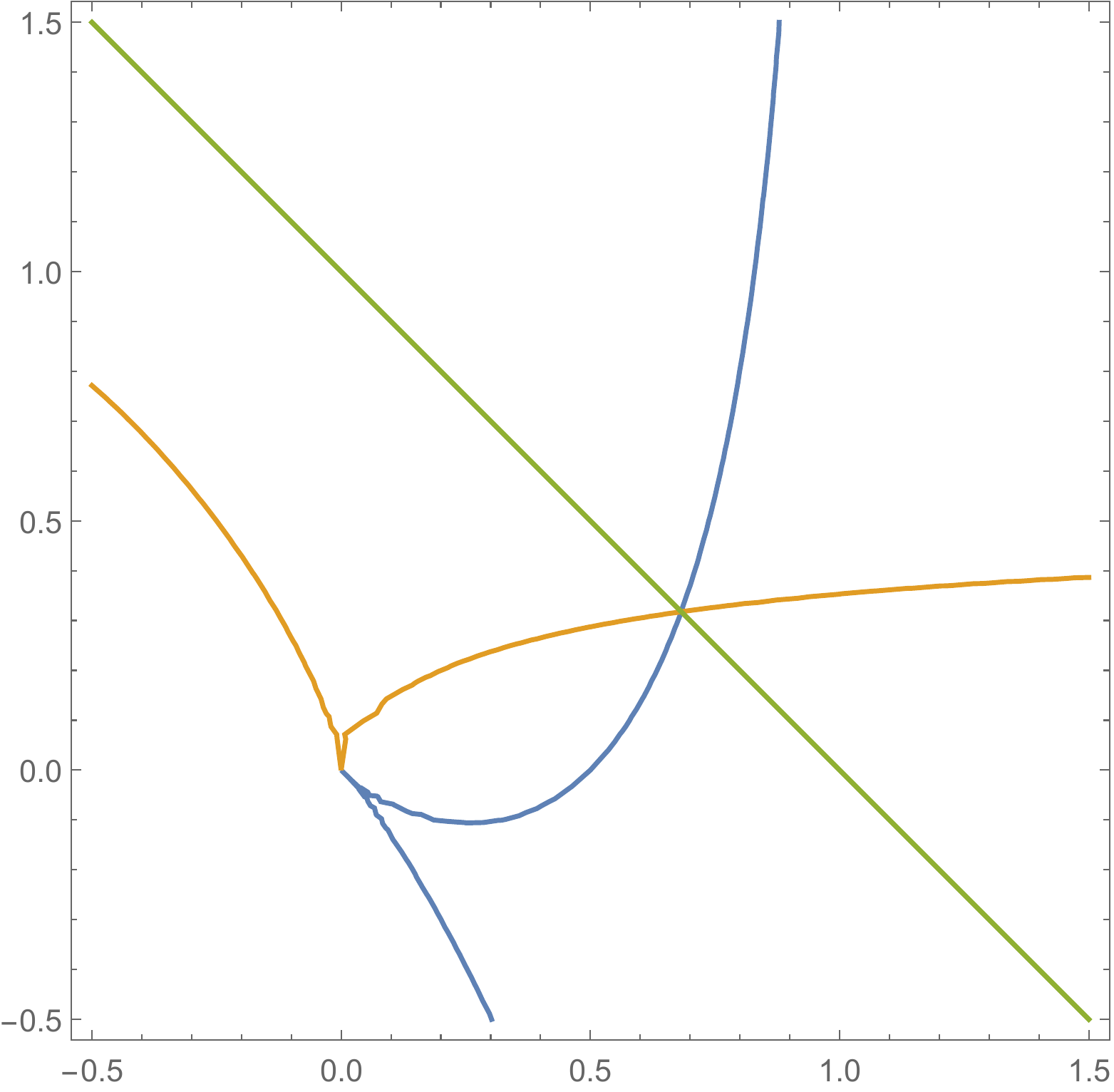}
		\caption{The graph of the fixed point equations for the golden mean}
		\label{fig:GMCurve}
	\end{center}
\end{figure}


\begin{remark}
	When estimating the entropy $h^{(k)}$ of $Z_M$ numerically, 
iteration of the mapping $T_k$ combined with a recursion on $\log |x(n)|$ can reduce the size of the numbers involved and give us more accurate estimates more quickly. 
Since
\be
\begin{aligned}
	x_i(n+1)&=|x(n)|^k(Mr(n))_i^k, i=0, \dots, d-1,\\
	|x(n+1)|&=|x(n)|^k g_k(r(n)),
\end{aligned}
\en
we may form the $T_kr(n)$ by iterating $T_k$ and use them in the recursion
\be
\log |x(n+1)| = k \log |x(n)| + \log g_k(r(n))
\en
to form rapidly improving approximations to
\be
h = \lim_{n \to \infty} \frac{\log|x(n)|}{1+k+ \dots +k^n}.
\en
The factors $g_k(r(n))$ are bounded, but they have a cumulative effect on the growth of $|x(n)|$ that may be sufficient to affect the ultimate value of $h$.
\end{remark}

 {Let $C(-1)=M$ and for each $n \geq 0$ define a $d \times d$ matrix $C(n)$ by}
 \be
 C(n)_{ij}=M_{ij}(M_i \cdot x(n))^{k-1}.
 \en
 Denote by $\lambda_n$ the maximal eigenvalue of $C(n)$ and define $c_n$ by
\be
\lambda_n=|x(n)|^{k-1}c_n.
\en
The entropy of the one-dimensional $n$'th strip approximation subshift to $Z_M$ is 
\be
{h_n^{(k)}=\frac{\log \lambda_{n-1}}{k^n}.}
\en

\begin{theorem}\label{th:genent}
	With notation as above, the entropy of the $k$-tree shift corresponding to an irreducible $d \times d$ $0,1$ matrix $M$ is given by the following infinite series formula:
	\be
	h^{(k)}(Z_M)=\frac{k-1}{k}\log d+(k-1)\sum_{i=1}^{\infty}\frac{1}{k^{i+1}}
 	\log g_k(r(i-1)).
\en
\end{theorem}
\begin{proof}
We compute that
\be
h_{n+1}^{(k)}=h_n^{(k)}+\frac{k-1}{k^{n+1}}\log g_k(r(n-1))+\frac{1}{k^{n+1}}
\log c_{n}-\frac{1}{k^n}\log c_{n-1},
\en
\be
h_1^{(k)}= \frac{k-1}{k}\log d+\frac{1}{k}\log(c_0), 
\en
\be
h_2^{(k)} = \frac{k-1}{k}\log d+\frac{k-1}{k^2}\log g_k(r(0))+\frac{1}{k^2}\log c_1, 
\en
and then by induction that
\be
h_n^{(k)}=\frac{k-1}{k}\log d+(k-1)\sum_{i=1}^{n-2}
\frac{1}{k^{i+1}}\log g_k(r(i-1))+\frac{1}{k^n}\log c_{n-1}.
\en
\end{proof}

We showed before that for the golden mean SFT on the $k$-tree, the site specific entropies $h_n^{(k)}$ increase with the strip width to the tree shift entropy $h^{(k)}$ (at least for $k=2,\dots, 8$), and that $h^{(k)}$ is strictly increasing in $k$, with limit $\log 2$, in contrast with the situation for the golden mean SFT's on integer lattices.
 The same statement for the more general tree shifts considered in this section can be approached by the same techniques, although the formulas and computations will naturally be much more complex. 
 Extensions to pressure and equilibrium states, including measures of maximal entropy, and subshifts on other trees and graphs, are also attractive topics for further research.

\begin{ack*}
	We thank Professors Henk Bruin, Kevin McGoff, Tom Meyerovitch, Anthony Quas, and Ville Salo for helpful discussions on these topics.
\end{ack*}

\section{Appendix}\label{sec:appendix}
{Here are the algebraic calculations to show that $h_n^{(k)}$ is strictly increasing with $n$ for the tree  dimensions $k=5,6,7,8$. 
See Section \ref{sec:mon} for the definitions.}

	\be
\begin{aligned}
	n_5(r)&=-1+\frac{5 r^4}{2}-6 r^5+15 r^8-10 r^9-9 r^{10}+25 r^{12}-10 r^{13}-15 r^{14}-\\
	&6 r^{15}+\frac{25 r^{16}}{2}-15 r^{18}-10 r^{19}-\frac{r^{20}}{2}-10
	r^{23}-\\
	&\frac{5 r^{24}}{2}-\frac{5 r^{28}}{2}+\frac{5}{2} r^4 \sqrt{1+4 r^4}-2 r^5 \sqrt{1+4 r^4}+\\
	&10 r^8 \sqrt{1+4 r^4}-6 r^9 \sqrt{1+4 r^4}-3 r^{10}\\
	&\sqrt{1+4 r^4}+10 r^{12} \sqrt{1+4 r^4}-2 r^{13} \sqrt{1+4 r^4}-9 r^{14} \sqrt{1+4 r^4}-\\
	&2 r^{15} \sqrt{1+4 r^4}+\frac{5}{2} r^{16} \sqrt{1+4 r^4}-3
	r^{18} \sqrt{1+4 r^4}-6 r^{19} \sqrt{1+4 r^4}-\\
	&\frac{1}{2} r^{20} \sqrt{1+4 r^4}-2 r^{23} \sqrt{1+4 r^4}-\frac{3}{2} r^{24} \sqrt{1+4 r^4}-
	\frac{1}{2}	r^{28} \sqrt{1+4 r^4}   \\
	A_5(r)&= -1+\frac{5 r^4}{2}-6 r^5+15 r^8-10 r^9-9 r^{10}+25 r^{12}-10 r^{13}-15 r^{14}-\\
	&6 r^{15}+\frac{25 r^{16}}{2}-15 r^{18}-10 r^{19}-\frac{r^{20}}{2}-10 r^{23}-\frac{5 r^{24}}{2}-\frac{5 r^{28}}{2}   \\
	B_5(r)&= \frac{5}{2} r^4 \sqrt{1+4 r^4}-2 r^5 \sqrt{1+4 r^4}+10 r^8 \sqrt{1+4 r^4}-6 r^9 \sqrt{1+4 r^4}-\\
	&3 r^{10} \sqrt{1+4 r^4}+10 r^{12} \sqrt{1+4 r^4}-2 r^{13} \sqrt{1+4 r^4}-9 r^{14} \sqrt{1+4 r^4}-\\
	&2 r^{15} \sqrt{1+4 r^4}+\frac{5}{2} r^{16} \sqrt{1+4 r^4}-3 r^{18} \sqrt{1+4 r^4}-6 r^{19} \sqrt{1+4 r^4}-\\
	&\frac{1}{2} r^{20} \sqrt{1+4 r^4}-2 r^{23} \sqrt{1+4 r^4}-\frac{3}{2} r^{24} \sqrt{1+4 r^4}-\frac{1}{2} r^{28} \sqrt{1+4 r^4}   \\
	d_5(r)&= 1-5 r^4+12 r^5-30 r^8+50 r^{10}-50 r^{12}-100 r^{13}+80 r^{14}+108 r^{15}-25 r^{16}-\\
	&200 r^{17}-70 r^{18}+240 r^{19}+137 r^{20}-100 r^{21}-300	r^{22}+140 r^{23}+345 r^{24}+100 r^{25}-\\
	&150 r^{26}-200 r^{27}+320 r^{28}+300 r^{29}+22 r^{30}-100 r^{31}-50 r^{32}+300 r^{33}+170 r^{34}-44 r^{35}-\\
	&25	r^{36}+170 r^{38}+60 r^{39}-68 r^{40}+60 r^{43}+10 r^{44}-56 r^{45}+10 r^{48}-28 r^{50}-8 r^{55}-r^{60}   \\
	q_5(r)&=  1+2 r+3 r^2+4 r^3+8 r^5+18 r^6+30 r^7+14 r^8+28 r^{10}+72 r^{11}+\\
	&85 r^{12}+20 r^{13}+56 r^{15}+143 r^{16}+110 r^{17}+15 r^{18}+68 r^{20}+\\
	&148r^{21}+74 r^{22}+2 r^{23}+48 r^{25}+82 r^{26}+20 r^{27}-13 r^{28}+16 r^{30}+28 r^{31}-\\
	&10 r^{32}-20 r^{33}+7 r^{36}-10 r^{37}-15 r^{38}-2 r^{42}-6r^{43}-r^{48} \\
	&={1+2 r+3 r^2+4 r^3+8 r^5+13r^6(1-r^{22})+5r^6+10r^7(1-r^{25})+20r^7+}\\
		&{20r^{10}(1-r^{23}) +8r^{10}+10r^{11}(1-r^{26})+62r^{11}+15r^{12}(1-r^{26})+70r^{12}+}\\
		&{2r^{13}(1-r^{29})+18r^{13}+6r^{15}(1-r^{28})+50r^{15}+r^{16}(1-r^{32})+142r^{16}+110r^{17}+}\\
		&{15 r^{18}+68 r^{20}+148r^{21}+74 r^{22}+2 r^{23}+48 r^{25}+82 r^{26}+20 r^{27}+}\\
		&{14r^8+16r^{30}+28r^{31}+7r^{36}}
\end{aligned}
\en
\be
\begin{aligned}
	n_6(r)&= -1+3 r^5-\frac{15 r^6}{2}+\frac{45 r^{10}}{2}-15 r^{11}-15 r^{12}+55 r^{15}-\\
	&\frac{45 r^{16}}{2}-30 r^{17}-15 r^{18}+\frac{105 r^{20}}{2}-\\
	&5	r^{21}-45 r^{22}-30 r^{23}-\frac{15 r^{24}}{2}+18 r^{25}-10 r^{27}-45 r^{28}-\\
	&15 r^{29}-\frac{r^{30}}{2}-10 r^{33}-\frac{45 r^{34}}{2}-3 r^{35}-\\
	&5	r^{39}-\frac{9 r^{40}}{2}-r^{45}+3 r^5 \sqrt{1+4 r^5}-\frac{5}{2} r^6 \sqrt{1+4 r^5}+\\
	&\frac{33}{2} r^{10} \sqrt{1+4 r^5}-10 r^{11} \sqrt{1+4 r^5}-5	r^{12} \sqrt{1+4 r^5}+\\
	&28 r^{15} \sqrt{1+4 r^5}-\frac{15}{2} r^{16} \sqrt{1+4 r^5}-20 r^{17} \sqrt{1+4 r^5}-5 r^{18} \sqrt{1+4 r^5}+\\
	&\frac{35}{2} r^{20}	\sqrt{1+4 r^5}-15 r^{22} \sqrt{1+4 r^5}-20 r^{23} \sqrt{1+4 r^5}-\\
	&\frac{5}{2} r^{24} \sqrt{1+4 r^5}+3 r^{25} \sqrt{1+4 r^5}-15 r^{28} \sqrt{1+4 r^5}-\\
	&10	r^{29} \sqrt{1+4 r^5}-\frac{1}{2} r^{30} \sqrt{1+4 r^5}-\frac{15}{2} r^{34} \sqrt{1+4 r^5}-\\
	&2 r^{35} \sqrt{1+4 r^5}-\frac{3}{2} r^{40} \sqrt{1+4 r^5}  \\
	A_6(r)&= -1+3 r^5-\frac{15 r^6}{2}+\frac{45 r^{10}}{2}-15 r^{11}-15 r^{12}+55 r^{15}-\frac{45 r^{16}}{2}-30 r^{17}-\\
	&15 r^{18}+\frac{105 r^{20}}{2}-5 r^{21}-45 r^{22}-30 r^{23}-\frac{15 r^{24}}{2}+18 r^{25}-10 r^{27}-\\
	&45 r^{28}-15 r^{29}-\frac{r^{30}}{2}-10 r^{33}-\frac{45 r^{34}}{2}-3 r^{35}-5 r^{39}-\frac{9 r^{40}}{2}-r^{45}   \\
	B_6(r)&= 3 r^5 \sqrt{1+4 r^5}-\frac{5}{2} r^6 \sqrt{1+4 r^5}+\frac{33}{2} r^{10} \sqrt{1+4 r^5}-10 r^{11} \sqrt{1+4 r^5}-\\
	&5 r^{12} \sqrt{1+4 r^5}+28 r^{15} \sqrt{1+4 r^5}-\frac{15}{2} r^{16} \sqrt{1+4 r^5}-20 r^{17} \sqrt{1+4 r^5}-\\
	&5 r^{18} \sqrt{1+4 r^5}+\frac{35}{2} r^{20} \sqrt{1+4 r^5}-15 r^{22} \sqrt{1+4 r^5}-20 r^{23} \sqrt{1+4 r^5}-\\
	&\frac{5}{2} r^{24} \sqrt{1+4 r^5}+3 r^{25} \sqrt{1+4 r^5}-15 r^{28} \sqrt{1+4 r^5}-10 r^{29} \sqrt{1+4 r^5}-\\
	&\frac{1}{2} r^{30} \sqrt{1+4 r^5}-\frac{15}{2} r^{34} \sqrt{1+4 r^5}-2 r^{35} \sqrt{1+4 r^5}-\frac{3}{2} r^{40} \sqrt{1+4 r^5}   
\end{aligned}
\en
\be
\begin{aligned}
	d_6(r)&=  1-6 r^5+15 r^6-45 r^{10}+80 r^{12}-110 r^{15}-180 r^{16}+150 r^{17}+230 r^{18}-\\
	&105 r^{20}-540 r^{21}-135 r^{22}+600 r^{23}+415 r^{24}-36	r^{25}-525 r^{26}-1030 r^{27}+\\
	&540 r^{28}+\1200 r^{29}+501 r^{30}-180 r^{31}-1050 r^{32}-\\
	&880 r^{33}+1620 r^{34}+1494 r^{35}+415 r^{36}-360 r^{37}-1050r^{38}-\\
	&140 r^{39}+2205 r^{40}+1230 r^{41}+255 r^{42}-360 r^{43}-525 r^{44}+390 r^{45}+1845 r^{46}+\\
	&660 r^{47}+180 r^{48}-180 r^{49}-105 r^{50}+410r^{51}+990 r^{52}+210 r^{53}+215 r^{54}-\\
	&36 r^{55}+220 r^{57}+315 r^{58}+30 r^{59}+252 r^{60}+70 r^{63}+45 r^{64}+210 r^{66}+10 r^{69}+\\
	&120 r^{72}+45	r^{78}+10 r^{84}+r^{90}  \\
	q_6(r)&=  1+2 r+3 r^2+4 r^3+5 r^4+10 r^6+22 r^7+36 r^8+52 r^9+25 r^{10}+45 r^{12}+110 r^{13}+\\
	&198 r^{14}+202 r^{15}+55 r^{16}+120 r^{18}+330 r^{19}+555r^{20}+394 r^{21}+70 r^{22}+210 r^{24}+\\
	&624 r^{25}+888 r^{26}+462 r^{27}+56 r^{28}+250 r^{30}+740 r^{31}+888 r^{32}+354 r^{33}+33 r^{34}+\\
	&200 r^{36}+540	r^{37}+600 r^{38}+192 r^{39}+33 r^{40}+100 r^{42}+240 r^{43}+285 r^{44}+100 r^{45}+\\
	&56 r^{46}+25 r^{48}+70 r^{49}+90 r^{50}+70 r^{51}+70 r^{52}+14	r^{55}+18 r^{56}+42 r^{57}+56 r^{58}+\\
	&3 r^{62}+14 r^{63}+28 r^{64}+2 r^{69}+8 r^{70}+r^{76}  
\end{aligned}
\en
\be
\begin{aligned}
	n_7(r)&=-1+\frac{7 r^6}{2}-9 r^7+\frac{63 r^{12}}{2}-21 r^{13}-\frac{45 r^{14}}{2}+\\
	&\frac{203 r^{18}}{2}-42 r^{19}-\frac{105 r^{20}}{2}-30 r^{21}+147
	r^{24}-\\
	&21 r^{25}-105 r^{26}-70 r^{27}-\frac{45 r^{28}}{2}+98 r^{30}-\frac{105 r^{32}}{2}-140 r^{33}-\\
	&\frac{105 r^{34}}{2}-9 r^{35}+\frac{49 r^{36}}{2}-70
	r^{39}-105 r^{40}-21 r^{41}-\frac{r^{42}}{2}-\\
	&\frac{105 r^{46}}{2}-42 r^{47}-\frac{7 r^{48}}{2}-21 r^{53}-7 r^{54}-\frac{7 r^{60}}{2}+\\
	&\frac{7}{2}r^6 \sqrt{1+4 r^6}-3 r^7 \sqrt{1+4 r^6}+\frac{49}{2} r^{12} \sqrt{1+4 r^6}-\\
	&15 r^{13} \sqrt{1+4 r^6}-\frac{15}{2} r^{14} \sqrt{1+4 r^6}+\frac{119}{2}
	r^{18} \\
	&\sqrt{1+4 r^6}-18 r^{19} \sqrt{1+4 r^6}-\frac{75}{2} r^{20} \sqrt{1+4 r^6}-10 r^{21} \sqrt{1+4 r^6}+\\
	&63 r^{24} \sqrt{1+4 r^6}-3 r^{25} \sqrt{1+4	r^6}-45 r^{26} \sqrt{1+4 r^6}-50 r^{27} \sqrt{1+4 r^6}-\\
	&\frac{15}{2} r^{28} \sqrt{1+4 r^6}+28 r^{30} \sqrt{1+4 r^6}-\frac{15}{2} r^{32} \sqrt{1+4	r^6}-\\
	&60 r^{33} \sqrt{1+4 r^6}-\frac{75}{2} r^{34} \sqrt{1+4 r^6}-3 r^{35} \sqrt{1+4 r^6}+\frac{7}{2} r^{36} \sqrt{1+4 r^6}-\\
	&10 r^{39} \sqrt{1+4 r^6}-45r^{40} \sqrt{1+4 r^6}-15 r^{41} \sqrt{1+4 r^6}-\frac{1}{2} r^{42} \sqrt{1+4 r^6}-\\
	&\frac{15}{2} r^{46} \sqrt{1+4 r^6}-18 r^{47} \sqrt{1+4 r^6}-\frac{5}{2}r^{48} \sqrt{1+4 r^6}-\\
	&3 r^{53} \sqrt{1+4 r^6}-3 r^{54} \sqrt{1+4 r^6}-\frac{1}{2} r^{60} \sqrt{1+4 r^6}
\end{aligned}
\en
\be
\begin{aligned}
	A_7(r)&=  -1+\frac{7 r^6}{2}-9 r^7+\frac{63 r^{12}}{2}-21 r^{13}-\frac{45 r^{14}}{2}+\frac{203 r^{18}}{2}-\\
	&42 r^{19}-\frac{105 r^{20}}{2}-30 r^{21}+147 r^{24}-21 r^{25}-105 r^{26}-70 r^{27}-\\
	&\frac{45 r^{28}}{2}+98 r^{30}-\frac{105 r^{32}}{2}-140 r^{33}-\frac{105 r^{34}}{2}-9 r^{35}+\\
	&\frac{49 r^{36}}{2}-70 r^{39}-105 r^{40}-21 r^{41}-\frac{r^{42}}{2}-\frac{105 r^{46}}{2}-\\
	&42 r^{47}-\frac{7 r^{48}}{2}-21 r^{53}-7 r^{54}-\frac{7 r^{60}}{2}  
\end{aligned}
\en
\be
\begin{aligned}
	B_7(r)&= \frac{7}{2} r^6 \sqrt{1+4 r^6}-3 r^7 \sqrt{1+4 r^6}+\frac{49}{2} r^{12} \sqrt{1+4 r^6}-\\
	&15 r^{13} \sqrt{1+4 r^6}-\frac{15}{2} r^{14} \sqrt{1+4 r^6}+\frac{119}{2} r^{18} \sqrt{1+4 r^6}-\\
	&18 r^{19} \sqrt{1+4 r^6}-\frac{75}{2} r^{20} \sqrt{1+4 r^6}-10 r^{21} \sqrt{1+4 r^6}+63 r^{24} \sqrt{1+4 r^6}-\\
	&3 r^{25} \sqrt{1+4 r^6}-45 r^{26} \sqrt{1+4 r^6}-50 r^{27} \sqrt{1+4 r^6}-\frac{15}{2} r^{28} \sqrt{1+4 r^6}+\\
	&28 r^{30} \sqrt{1+4 r^6}-\frac{15}{2} r^{32} \sqrt{1+4 r^6}-60 r^{33} \sqrt{1+4 r^6}-\frac{75}{2} r^{34} \sqrt{1+4 r^6}-\\
	&3 r^{35} \sqrt{1+4 r^6}+\frac{7}{2} r^{36} \sqrt{1+4 r^6}-10 r^{39} \sqrt{1+4 r^6}-45 r^{40} \sqrt{1+4 r^6}-\\
	&15 r^{41} \sqrt{1+4 r^6}-\frac{1}{2} r^{42} \sqrt{1+4 r^6}-\frac{15}{2} r^{46} \sqrt{1+4 r^6}-\\
	&18 r^{47} \sqrt{1+4 r^6}-\frac{5}{2} r^{48} \sqrt{1+4 r^6}-3 r^{53} \sqrt{1+4 r^6}-\\
	&3 r^{54} \sqrt{1+4 r^6}-\frac{1}{2} r^{60} \sqrt{1+4 r^6}   
\end{aligned}
\en
\be
\begin{aligned}
	d_7(r)&=  1-7 r^6+18 r^7-63 r^{12}+117 r^{14}-203 r^{18}-294 r^{19}+252 r^{20}+420 r^{21}-294 r^{24}-\\
	&1176 r^{25}-231 r^{26}+1260 r^{27}+975 r^{28}-196	r^{30}-1764 r^{31}-2688 r^{32}+1540 r^{33}+\\
	&3255 r^{34}+1578 r^{35}-49 r^{36}-1176 r^{37}-4410 r^{38}-2660 r^{39}+5775 r^{40}+5460 r^{41}+\\
	&1845 r^{42}-294r^{43}-2940 r^{44}-5880 r^{45}+315 r^{46}+10626 r^{47}+6461 r^{48}+1566 r^{49}-\\
	&735 r^{50}-3920 r^{51}-4410 r^{52}+4284 r^{53}+12873 r^{54}+5586 r^{55}+909r^{56}-980 r^{57}-\\
	&2940 r^{58}-1764 r^{59}+6265 r^{60}+11172 r^{61}+3570 r^{62}+200 r^{63}-735 r^{64}-\\
	&1176 r^{65}-294r^{66}+5586 r^{67}+7140 r^{68}+1680	r^{69}-378 r^{70}-294 r^{71}-196 r^{72}+\\
	&3570 r^{74}+3360 r^{75}+567 r^{76}-774 r^{77}-49 r^{78}+1680 r^{81}+1134 r^{82}+126 r^{83}-922 r^{84}+\\
	&567r^{88}+252 r^{89}+14 r^{90}-792 r^{91}+126 r^{95}+28 r^{96}-495 r^{98}+14 r^{102}-220 r^{105}-\\
	&66 r^{112}-12 r^{119}-r^{126}  
\end{aligned}
\en
{\be
\begin{aligned}
	q_7(r)&=  1+2 r+3 r^2+4 r^3+5 r^4+6 r^5+12 r^7+26 r^8+42 r^9+60 r^{10}+80 r^{11}+39 r^{12}+\\
	&66 r^{14}+156 r^{15}+273 r^{16}+420 r^{17}+397 r^{18}+116	r^{19}+220 r^{21}+572 r^{22}+\\
	&1092 r^{23}+1526 r^{24}+1036 r^{25}+209 r^{26}+495 r^{28}+1430 r^{29}+2807 r^{30}+3304 r^{31}+\\
	&1708 r^{32}+252 r^{33}+792	r^{35}+2525 r^{36}+4732 r^{37}+4711 r^{38}+1932 r^{39}+210 r^{40}+\\
	&922 r^{42}+3134 r^{43}+5377 r^{44}+4724 r^{45}+1537 r^{46}+114 r^{47}+780 r^{49}+2669	r^{50}+\\
	&4270 r^{51}+3422 r^{52}+820 r^{53}+6 r^{54}+465 r^{56}+1510 r^{57}+2471 r^{58}+1736 r^{59}+\\
	&193 r^{60}-106 r^{61}+180 r^{63}+555 r^{64}+1064r^{65}+532 r^{66}-154 r^{67}-208 r^{68}+\\
	&36 r^{70}+138 r^{71}+329 r^{72}+28 r^{73}-238 r^{74}-252 r^{75}+23 r^{78}+70 r^{79}-56 r^{80}-\\
	&168 r^{81}-210	r^{82}+10 r^{86}-24 r^{87}-72 r^{88}-120 r^{89}-3 r^{94}-18 r^{95}-45 r^{96}-\\
	&2 r^{102}-10 r^{103}-r^{110}\\
	&=  1+2 r+3 r^2+4 r^3+5 r^4+6 r^5+12 r^7+26 r^8+42 r^9+60 r^{10}+80 r^{11}+39 r^{12}+66 r^{14}+\\
	&154r^{15}(1-r^{52})+2r^15+208r^16(1-r^{52})+65r^{16}+238r^{17}(1-r^{57})+182r^{17}+\\
	&252r^{18}(1-r^{57})+145r^{18}+56r^{19}(1-r^{61})+60r^{19}+168r^{21}(1-r^{60})+52r^{21}+\\
	&210r^{22}(1-r^{60})+362r^{22}+24r^{23}(1-r^{64})+1068r^{23}+72r^{24}(1-r^{64})+1454r^{24}+\\
	&120r^{25}(1-r^{64})+916r^{25}+3r^{26}(1-r^{68})+206r^{26}+18r^{28}(1-r^{67})+477r^{28}+\\
	&45r^{29}(1-r^{67})+1385r^{29}+2r^{30}(1-r^{72})+2805r^{30}+10r^{31}(1-r^{72})+3294r^{31}+\\
	&r^{32}(1-r^{78})+1707r^{32}+106r^{33}(1-r^{28})+146r^33+\\
	&792	r^{35}+2525 r^{36}+4732 r^{37}+4711 r^{38}+1932 r^{39}+210 r^{40}+\\
	&922 r^{42}+3134 r^{43}+5377 r^{44}+4724 r^{45}+1537 r^{46}+114 r^{47}+780 r^{49}+2669	r^{50}+\\
	&4270 r^{51}+3422 r^{52}+820 r^{53}+6 r^{54}+465 r^{56}+1510 r^{57}+2471 r^{58}+1736 r^{59}+\\
	&193 r^{60}+180 r^{63}+555 r^{64}+1064r^{65}+532 r^{66}+36 r^{70}+138 r^{71}+329 r^{72}+28 r^{73}+\\
	&23r^{78}+70r^{79}+10r^{86}
\end{aligned}
\en}
\be
\begin{aligned}
	n_8(r)&=  -1+4 r^7-\frac{21 r^8}{2}+42 r^{14}-28 r^{15}-\frac{63 r^{16}}{2}+168 r^{21}-70 r^{22}-\\
	&84 r^{23}-\frac{105 r^{24}}{2}+329 r^{28}-56 r^{29}-210	r^{30}-140 r^{31}-\frac{105 r^{32}}{2}+\\
	&336 r^{35}-7 r^{36}-168 r^{37}-350 r^{38}-140 r^{39}-\frac{63 r^{40}}{2}+168 r^{42}-21 r^{44}-\\
	&280 r^{45}-350	r^{46}-84 r^{47}-\frac{21 r^{48}}{2}+32 r^{49}-35 r^{52}-280 r^{53}-210 r^{54}-\\
	&28 r^{55}-\frac{r^{56}}{2}-35 r^{60}-168 r^{61}-70 r^{62}-4 r^{63}-21	r^{68}-56 r^{69}-10 r^{70}-\\
	&7 r^{76}-8 r^{77}-r^{84}+4 r^7 \sqrt{1+4 r^7}-\frac{7}{2} r^8 \sqrt{1+4 r^7}+34 r^{14} \sqrt{1+4 r^7}-\\
	&21 r^{15} \sqrt{1+4r^7}-\frac{21}{2} r^{16} \sqrt{1+4 r^7}+108 r^{21} \sqrt{1+4 r^7}-\\
	&35 r^{22} \sqrt{1+4 r^7}-63 r^{23} \sqrt{1+4 r^7}-\frac{35}{2} r^{24} \sqrt{1+4
		r^7}+\\
	&165 r^{28} \sqrt{1+4 r^7}-14 r^{29} \sqrt{1+4 r^7}-105 r^{30} \sqrt{1+4 r^7}-105 r^{31} \sqrt{1+4 r^7}-\\
	&\frac{35}{2} r^{32} \sqrt{1+4 r^7}+126	r^{35} \sqrt{1+4 r^7}-42 r^{37} \sqrt{1+4 r^7}-\\
	&175 r^{38} \sqrt{1+4 r^7}-105 r^{39} \sqrt{1+4 r^7}-\frac{21}{2} r^{40} \sqrt{1+4 r^7}+\\
	&42 r^{42} \sqrt{1+4r^7}-70 r^{45} \sqrt{1+4 r^7}-175 r^{46} \sqrt{1+4 r^7}-63 r^{47} \sqrt{1+4 r^7}-\\
	&\frac{7}{2} r^{48} \sqrt{1+4 r^7}+4 r^{49} \sqrt{1+4 r^7}-70 r^{53}	\sqrt{1+4 r^7}-\\
	&105 r^{54} \sqrt{1+4 r^7}-21 r^{55} \sqrt{1+4 r^7}-\frac{1}{2} r^{56} \sqrt{1+4 r^7}-\\
	&42 r^{61} \sqrt{1+4 r^7}-35 r^{62} \sqrt{1+4	r^7}-3 r^{63} \sqrt{1+4 r^7}-14 r^{69} \sqrt{1+4 r^7}-\\
	&5 r^{70} \sqrt{1+4 r^7}-2 r^{77} \sqrt{1+4 r^7} 
\end{aligned}
\en
\be
\begin{aligned}
	A_8(r)&= -1+4 r^7-\frac{21 r^8}{2}+42 r^{14}-28 r^{15}-\frac{63 r^{16}}{2}+168 r^{21}-70 r^{22}-\\
	&84 r^{23}-\frac{105 r^{24}}{2}+329 r^{28}-56 r^{29}-210r^{30}-140 r^{31}-\frac{105 r^{32}}{2}+\\
	&336 r^{35}-7 r^{36}-168 r^{37}-350 r^{38}-140 r^{39}-\frac{63 r^{40}}{2}+168 r^{42}-21 r^{44}-\\
	&280 r^{45}-350	r^{46}-84 r^{47}-\frac{21 r^{48}}{2}+32 r^{49}-35 r^{52}-280 r^{53}-210 r^{54}-\\
	&28 r^{55}-\frac{r^{56}}{2}-35 r^{60}-168 r^{61}-70 r^{62}-4 r^{63}-21r^{68}-56 r^{69}-10 r^{70}-\\
	&7 r^{76}-8 r^{77}-r^{84}   
\end{aligned}
\en
\be
\begin{aligned}
	B_8(r)&= 4 r^7 \sqrt{1+4 r^7}-\frac{7}{2} r^8 \sqrt{1+4 r^7}+34 r^{14} \sqrt{1+4 r^7}-21 r^{15} \sqrt{1+4 r^7}-\\
	&\frac{21}{2} r^{16} \sqrt{1+4 r^7}+108 r^{21} \sqrt{1+4 r^7}-35 r^{22} \sqrt{1+4 r^7}-\\
	&63 r^{23} \sqrt{1+4 r^7}-\frac{35}{2} r^{24} \sqrt{1+4 r^7}+165 r^{28} \sqrt{1+4 r^7}-\\
	&14 r^{29} \sqrt{1+4 r^7}-105 r^{30} \sqrt{1+4 r^7}-105 r^{31} \sqrt{1+4 r^7}-\\
	&\frac{35}{2} r^{32} \sqrt{1+4 r^7}+126 r^{35} \sqrt{1+4 r^7}-42 r^{37} \sqrt{1+4 r^7}-\\
	&175 r^{38} \sqrt{1+4 r^7}-105 r^{39} \sqrt{1+4 r^7}-\frac{21}{2} r^{40} \sqrt{1+4 r^7}+\\
	&42 r^{42} \sqrt{1+4 r^7}-70 r^{45} \sqrt{1+4 r^7}-175 r^{46} \sqrt{1+4 r^7}-63 r^{47} \sqrt{1+4 r^7}-\\
	&\frac{7}{2} r^{48} \sqrt{1+4 r^7}+4 r^{49} \sqrt{1+4 r^7}-70 r^{53} \sqrt{1+4 r^7}-105 r^{54} \sqrt{1+4 r^7}-\\
	&21 r^{55} \sqrt{1+4 r^7}-\frac{1}{2} r^{56} \sqrt{1+4 r^7}-42 r^{61} \sqrt{1+4 r^7}-35 r^{62} \sqrt{1+4 r^7}-\\
	&3 r^{63} \sqrt{1+4 r^7}-14 r^{69} \sqrt{1+4 r^7}-5 r^{70} \sqrt{1+4 r^7}-2 r^{77} \sqrt{1+4 r^7}   
\end{aligned}
\en
\be
\begin{aligned}
	d_8(r)&= 1-8 r^7+21 r^8-84 r^{14}+161 r^{16}-336 r^{21}-448 r^{22}+392 r^{23}+693 r^{24}-658 r^{28}-\\
	&2240 r^{29}-364 r^{30}+2352 r^{31}+1967 r^{32}-672	r^{35}-4592 r^{36}-5936 r^{37}+3640 r^{38}+\\
	&7448 r^{39}+3983 r^{40}-336 r^{42}-4704 r^{43}-13678 r^{44}-6496 r^{45}+16380 r^{46}+15680 r^{47}+\\
	&5999r^{48}-64 r^{49}-2352 r^{50}-14112 r^{51}-22372 r^{52}+3696 r^{53}+37856 r^{54}+23912 r^{55}+\\
	&6861 r^{56}-448 r^{57}-7056 r^{58}-23520 r^{59}-21098r^{60}+24640 r^{61}+59332 r^{62}+27432 r^{63}+\\
	&5999 r^{64}-1344 r^{65}-11760 r^{66}-23520 r^{67}-9856 r^{68}+45584 r^{69}+68516 r^{70}+23968 r^{71}+\\
	&4032r^{72}-2240 r^{73}-11760 r^{74}-14112 r^{75}+1386 r^{76}+54544 r^{77}+59920 r^{78}+15848 r^{79}+\\
	&2261 r^{80}-2240 r^{81}-7056 r^{82}-4704 r^{83}+6202r^{84}+47936 r^{85}+39620 r^{86}+7728 r^{87}+\\
	&1624 r^{88}-1344 r^{89}-2352 r^{90}-672 r^{91}+5992 r^{92}+31696 r^{93}+19320 r^{94}+2632 r^{95}+\\
	&2121r^{96}-448 r^{97}-336 r^{98}+3962 r^{100}+15456 r^{101}+6580 r^{102}+560 r^{103}+3010 r^{104}-\\
	&64 r^{105}+1932 r^{108}+5264 r^{109}+1400 r^{110}+56r^{111}+3432 r^{112}+658 r^{116}+1120 r^{117}+\\
	&140 r^{118}+3003 r^{120}+140 r^{124}+112 r^{125}+2002 r^{128}+14 r^{132}+1001 r^{136}+364 r^{144}+\\
	&91	r^{152}+14 r^{160}+r^{168}   
\end{aligned}
\en
\be
\begin{aligned}
	q_8(r)&= 1+2 r+3 r^2+4 r^3+5 r^4+6 r^5+7 r^6+14 r^8+30 r^9+48 r^{10}+68 r^{11}+90 r^{12}+114 r^{13}+\\
	&56 r^{14}+91 r^{16}+210 r^{17}+360 r^{18}+544r^{19}+765 r^{20}+690 r^{21}+210 r^{22}+364 r^{24}+\\
	&910 r^{25}+1680 r^{26}+2720 r^{27}+3422 r^{28}+2258 r^{29}+490 r^{30}+1001 r^{32}+2730 r^{33}+\\
	&5460r^{34}+8848 r^{35}+9364 r^{36}+4804 r^{37}+791 r^{38}+2002 r^{40}+6006 r^{41}+12768 r^{42}+\\
	&19376 r^{43}+17402 r^{44}+7244 r^{45}+924 r^{46}+3003r^{48}+9946 r^{49}+21544 r^{50}+29824 r^{51}+\\
	&23488 r^{52}+8052 r^{53}+792 r^{54}+3430 r^{56}+12418 r^{57}+26234 r^{58}+33720 r^{59}+23900 r^{60}+\\
	&6672	r^{61}+502 r^{62}+2989 r^{64}+11494 r^{65}+23212 r^{66}+29096 r^{67}+18552 r^{68}+4128 r^{69}+\\
	&276 r^{70}+1960 r^{72}+7658 r^{73}+15176 r^{74}+19600r^{75}+10944 r^{76}+2024 r^{77}+276 r^{78}+\\
	&931 r^{80}+3542 r^{81}+7504 r^{82}+10304 ^{83}+4942 r^{84}+1108 r^{85}+502 r^{86}+294 r^{88}+\\
	&1106 r^{89}+2828r^{90}+4144 r^{91}+1876 r^{92}+980 r^{93}+792 r^{94}+49 r^{96}+238 r^{97}+\\
	&784 r^{98}+1232 r^{99}+770 r^{100}+924 r^{101}+924 r^{102}+34 r^{105}+152r^{106}+256 r^{107}+\\
	&360 r^{108}+660 r^{109}+792 r^{110}+19 r^{114}+36 r^{115}+135 r^{116}+330 r^{117}+495 r^{118}+\\
	&4 r^{123}+30 r^{124}+110 r^{125}+220r^{126}+3 r^{132}+22 r^{133}+66 r^{134}+2 r^{141}+12r^{142}+r^{150}   
\end{aligned}
\en

 \end{document}